\documentclass[a4paper,12pt]{amsart}
\usepackage{etex}
\usepackage{tikz, tikz-cd}

\usepackage{bm}
\usepackage{amssymb}
\usepackage{mathrsfs}
\usepackage{amsmath, amssymb,amsthm,latexsym,amscd,mathrsfs}
\usepackage{indentfirst}
\usepackage{stmaryrd}
\usepackage{graphicx}
\usepackage{subfigure}
\usepackage{extarrows}
\usepackage{amssymb}
\usepackage{amsmath,amssymb,amsthm,latexsym,amscd,mathrsfs}
\usepackage{indentfirst}
\usepackage{multirow}
\usepackage{latexsym}
\usepackage{amsfonts}
\usepackage{color}
\usepackage{pictexwd,dcpic}
\usepackage{graphicx}
\usepackage{psfrag}
\usepackage{hyperref}
\usepackage{comment}
\usepackage{cases}
\usepackage{dutchcal}

\makeatletter
\@namedef{subjclassname@2010}{{\rm2020} Mathematics Subject Classification}
\makeatother

 \newcommand{\ROM}[1]{\mathrm{\uppercase\expandafter{\romannumeral#1}}}
  \theoremstyle{definition}
   \numberwithin{equation}{section} \theoremstyle{plain}

 \newtheorem{thm}{Theorem}[section]
 \newtheorem{lem}{Lemma}[section]
 
 \newtheorem{cor}{Corollary}[section]
  
 \newtheorem{rem}{Remark}[section]
 \newtheorem{prop}{Proposition}[section]

  \makeatletter

  \numberwithin{equation}{section}
  \allowdisplaybreaks
\makeatother
\title[On two questions of James]{\textbf{On two questions of James}}
\author[Chao Qian]{Chao Qian}\address{School of Mathematics and Statistics, Beijing Institute of Technology, Beijing
100081, P.R. China}
\email{6120150035@bit.edu.cn}
\author[Z. Z. Tang]{Zizhou Tang}\address{Chern Institute of Mathematics $\&$ LPMC, Nankai University, Tianjin 300071, P. R. China}
\email{zztang@nankai.edu.cn}
\author[W. J. Yan]{Wenjiao Yan}\address{School of Mathematical Sciences, Laboratory of Mathematics and Complex Systems, Beijing Normal University, Beijing, 100875, P. R. China}
\email{wjyan@bnu.edu.cn}
\thanks {The project is partially supported by the NSFC (No.11871282, 11931007), BNSF (Z190003), Nankai Zhide Foundation and Beijing Institute of Technology Research Fund Program for Young Scholars.
}


\begin{document}

\maketitle

\begin{abstract}
64 years ago, I. M. James raised two fundamental questions about octonionic Stiefel spaces. The prime objective of this paper is to figure out
partial answers to both of them.
\end{abstract}

\section{\textbf{Introduction}}\label{sec1}
There are a number of fascinating and important properties of the Stiefel manifolds $V_k(\mathbb{R}^n)$, $V_k(\mathbb{C}^n)$ and $V_k(\mathbb{H}^n)$,
which are defined to be the set of all orthonormal $k$-frames in $\mathbb{R}^n$, $\mathbb{C}^n$, or $\mathbb{H}^n$ with respect to real, complex or quaternionic inner products. For example, they are compact smooth manifolds, and diffeomorphic to homogeneous spaces: $V_k(\mathbb{R}^n)\cong O(n)/O(n-k)$, $V_k(\mathbb{C}^n)\cong U(n)/U(n-k)$, $V_k(\mathbb{H}^n)\cong \mathrm{Sp}(n)/\mathrm{Sp}(n-k)$,  
respectively. 
In particular, $V_1(F^n)$ ($F=\mathbb{R}, \mathbb{C}, \mathbb{H}$) are unit spheres and
$V_2(F^n)$ can be regarded as sphere bundles over spheres. Moreover, there are natural fibre bundles $\pi: V_k(F^n)\rightarrow V_q(F^n)(q<k)$ by 
taking the last $q$ vectors of each $k$-frame as a $q$-frame. Furthermore, $V_k(F^n)$ is parallelizable whenever $k>1$ (cf. \cite{Sut64}).
The description of the values of $n$, $k$ and $q$ for which the fiber bundle $\pi$ admits
a cross-section has been widely studied(cf. \cite{Ja58}, \cite{Jam59}, \cite{Jam76}).

Comparing with $V_k(F^n)$ ($F=\mathbb{R}, \mathbb{C}, \mathbb{H}$), the octonionic Stiefel space has not been defined until 1958, when  I. M. James \cite{Ja58}
introduced $V_k(\mathbb{O}^n)$---the space of orthonormal $k$-frames in $\mathbb{O}^{n}$ with natural topology. Specifically, using the octonionic inner product $\langle \cdot, \cdot \rangle_{\mathbb{O}}$, 
he defined
\begin{equation}\label{james}
V_k(\mathbb{O}^n)=:\{(a_1,\cdots ,a_k)\in \mathbb{O}^n\oplus \cdots\oplus
\mathbb{O}^n~|~\langle a_i, a_j\rangle_{\mathbb{O}}=\delta_{ij}, \;1\leqslant i,j\leqslant k   \}.
\end{equation} 
However, although it can be defined, one still knows very little about $V_k(\mathbb{O}^n)$
due to the non-associativity of the multiplication in $\mathbb{O}$ and the lack of a transitive group action on the space for $k>1$ in general.

Similar to the real, complex and quaternion case, James defined a canonical projection
\begin{equation}\label{def pi}
\pi: V_{k}(\mathbb{O}^n)\rightarrow V_l(\mathbb{O}^n),\quad (1\leqslant l< k \leqslant n)
\end{equation}
 and asked the following fundamental \textbf{Questions}:
\begin{itemize}
\item[(1)] Is the projection $\pi$ a fiber map ?
\item[(2)] Is $V_k(\mathbb{O}^n)$ a manifold ? 
\end{itemize}

When $k=2$, James \cite{Ja58} proved that the projection $\pi: V_{2}(\mathbb{O}^n)\rightarrow V_1(\mathbb{O}^n)$ is a fiber map, hence $V_{2}(\mathbb{O}^n)$ admits a manifold structure. As the total space of a sphere bundle over sphere, $V_{2}(\mathbb{O}^n)$ is parallelizable (cf. \cite{Sut64}). We know that $V_2(\mathbb{O}^2)$ cannot be an $H$-space (cf. \cite{Jam61}, \cite{DS69}). Moreover,
according to \cite{Kra02}, for $n=2$, $V_2(\mathbb{O}^2)$ is a homogeneous space and diffeomorphic to $\mathrm{Spin}(9)/G_2$, and for $n\geqslant 3$, $V_{2}(\mathbb{O}^n)$ is not a homogeneous space.  James \cite{Ja58} found that a cross-section of the fiber bundle $\pi: V_{2}(\mathbb{O}^n)\rightarrow V_1(\mathbb{O}^n)$ occurs when $n=240$.
Furthermore, it was shown in \cite{QTY22} that
the sphere bundle $S^{8n-9}\hookrightarrow V_{2}(\mathbb{O}^n)\rightarrow V_1(\mathbb{O}^n)$ admits a cross-section if and only if $n$ can be divided by $240$.

The present paper is mainly concerned with the two fundamental questions of I. M. James on $V_k(\mathbb{O}^n)$.
Firstly, a partial answer to Question (1) is given as follows.

\begin{thm}\label{fiber}
The projection $\pi: V_{k+1}(\mathbb{O}^n)\rightarrow V_k(\mathbb{O}^n)$ $(n> k\geqslant 2)$ is not a fibration in the sense of Serre.
\end{thm}

We also make another partial progress to Question (1) in Proposition \ref{onto}, which will be discussed after we introduce Theorem \ref{omega}.

To attack Question (2),
since $V_1(\mathbb{O}^n)$ and $V_2(\mathbb{O}^n)$ are already known to be smooth manifolds, the direct and challenging problem is to ask whether
$V_3(\mathbb{O}^n)$ admits a manifold structure.
For this, we will introduce the spaces $\Omega_{l,m}$, which are generalizations of $V_3(\mathbb{O}^n)$ based on the representation theory of Clifford algebra. Let $E_1, \cdots, E_{m-1}$ be a representation of the Clifford algebra $\mathcal{C}_{m-1}$ on $\mathbb{R}^l$ with $l=n\delta(m)$, where $n$ is a positive integer and $\delta(m)$ is the dimension of the irreducible module of $\mathcal{C}_{m-1}$ valued as:

\begin{center}
\begin{tabular}{|c|c|c|c|c|c|c|c|c|cc|}
\hline
$m$ & 1 & 2 & 3 & 4 & 5 & 6 & 7 & 8 & $\cdots$ &$m$+8 \\
\hline
$\delta(m)$ & 1 & 2 & 4 & 4 & 8 & 8 & 8 & 8 &$\cdots$ &16$\delta(m)$\\
\hline
\end{tabular}
\end{center}
That is to say, $E_i$'s are skew-symmetric endomorphisms of $\mathbb{R}^l$ satisfying $E_iE_j+E_jE_i=-2\delta_{ij}Id$ for $i, j=1,\cdots, m-1$. We define
\begin{equation}\label{def omega}
\Omega_{l,m}:=\left\{(\textbf{a}, \textbf{b}, \textbf{c})\in \mathbb{R}^l\oplus\mathbb{R}^l\oplus\mathbb{R}^l~\vline~\begin{array}{lll}~~\,\,|\textbf{a}|^2=|\textbf{b}|^2=|\textbf{c}|^2=1, \vspace{1mm}\\
\langle \textbf{a}, \textbf{b} \rangle=\langle \textbf{b}, \textbf{c} \rangle=\langle \textbf{c}, \textbf{a} \rangle=0,\vspace{1mm}\\
\langle \textbf{a}, E_i\textbf{b} \rangle=\langle \textbf{b}, E_j\textbf{c} \rangle=\langle \textbf{c}, E_k\textbf{a} \rangle=0,\vspace{1mm}\\                 \,\,\,\,\,~\textrm{for}~i,j,k=1,\cdots, m-1.\end{array}\right\}.
\end{equation}
We emphasize that the space $\Omega_{l,m}$ depends not only on the values of $l$ and $m$, but also on the choice of the Clifford system $\{E_1,\cdots, E_{m-1}\}$.


As the second main result of this paper, we show
\begin{thm}\label{omega}
When it is not empty, $\Omega_{l, m}\subset \mathbb{R}^{3l}$ is a closed regular submanifold of dimension $3(l-m-1)$ with trivial normal bundle.
\end{thm}
As a direct consequence, we give a partial answer to Question (2) of James.
\begin{cor}\label{v3ok}
$V_3(\mathbb{O}^{n})$ is a smooth manifold of dimension $3(8n-9)$.
\end{cor}

To prove Theorem \ref{omega}, listing $F_i$ as the real functions corresponding to the functions $|\textbf{a}|^2-1$,  $|\textbf{b}|^2-1$,  $|\textbf{c}|^2-1$, $\langle \textbf{a}, \textbf{b} \rangle$, $\langle \textbf{b}, \textbf{c} \rangle$, $\langle \textbf{c}, \textbf{a} \rangle$,$\langle \textbf{a}, E_i\textbf{b} \rangle$, $\langle \textbf{b}, E_j\textbf{c} \rangle$ and $\langle \textbf{c}, E_k\textbf{a} \rangle$ in the definition of $\Omega_{l,m}$, we construct a smooth map
\begin{equation}\label{def F}
F: \mathbb{R}^{3l}\rightarrow \mathbb{R}^{3m+3}
\end{equation}
so that $\Omega_{l, m}=F^{-1}(0)$. To finish the proof of Theorem \ref{omega}, we will show that $0\in \mathbb{R}^{3m+3}$ is a regular value of $F$.


From another point of view, $\Omega_{l, m}$ can be regarded as a proper generalization of the focal submanifold $M_+$ of isoparametric hypersurfaces of OT-FKM type in the unit sphere $S^{2l-1}(1)$.
To express fluently, we give a short review of the isoparametric family of OT-FKM type. The isoparametric hypersurfaces in the unit sphere are those with constant principal curvatures. The number $g$ of distinct principal curvatures can be only $1, 2,3, 4$ or $6$. Among which, the isoparametric family with $g=4$ are the most complicated and with most abundant geometry (see for example, \cite{TY13}, \cite{TY15} and \cite{QTY21}). Meanwhile, except for two homogeneous cases, all the isoparametric families with $g=4$ are of OT-FKM type. By definition,
for a symmetric Clifford system $\{P_0,\cdots, P_m\}$ on $\mathbb{R}^{2l}$, \emph{i.e.}, $P_{i}$'s are symmetric matrices satisfying $P_{i}P_{j}+P_{j}P_{i}=2\delta_{ij}I_{2l}$, the following polynomial $P$ of degree $4$ on
$\mathbb{R}^{2l}$ is constructed (cf. \cite{OT75}, \cite{FKM81}):
\begin{eqnarray*}\label{FKM isop. poly.}
&&\qquad P:\quad \mathbb{R}^{2l}\rightarrow \mathbb{R}\nonumber\\
&&P(x) = |x|^4 - 2\displaystyle\sum_{i = 0}^{m}{\langle
P_{i}x,x\rangle^2}.
\end{eqnarray*}
The regular level sets of $f=P~|_{S^{2l-1}(1)}$ are called \emph{the isoparametric hypersurfaces of OT-FKM type} in $S^{2l-1}(1)$. The singular level sets $M_+=f^{-1}(1)$, $M_-=f^{-1}(-1)$ are called \emph{focal submanifolds}, which are proved to be minimal submanifolds of $S^{2l-1}(1)$ (cf. \cite{GT14}).
The multiplicities of the principal curvatures of an isoparametric hypersurface of OT-FKM type are $(m_1, m_2, m_1, m_2)$, where $(m_1, m_2)=(m, l-m-1)$ with
$l=n\delta(m)$. According to \cite{FKM81}, when $m\not \equiv 0 ~( \mathrm{mod}~4)$, there exists exactly one kind of OT-FKM type isoparametric family;
when $m\equiv 0~(\mathrm{mod} ~4)$, there are two kinds of OT-FKM type isoparametric families which are distinguished by $\mathrm{Trace}(P_0P_1\cdots P_m)$, namely, the family with $P_0P_1\cdots P_m=\pm I_{2l}$, 
which is called the \emph{definite} family, and the others with $P_0P_1\cdots P_m\neq\pm I_{2l}$
called \emph{indefinite}. There are exactly $[\frac{k}{2}]$ non-congruent indefinite families.

Since we can always express $P_0,\cdots, P_m$ up to an isomorphism as
\begin{equation*}\label{FKM}
P_0=\left(
\begin{matrix}
I_l & 0 \\
0 & -I_l \\
\end{matrix}\right),\;
P_1=\left(
\begin{matrix}
0 & I_l\\
I_l & 0 \\
\end{matrix}\right),
P_{i+1}=\left(
\begin{matrix}
0 & E_{i}\\
-E_{i} & 0 \\
\end{matrix}\right),
~\text{for}~ 1 \leqslant i\leqslant m-1,
\end{equation*}
where $E_1, \cdots, E_{m-1}$ are the generators of the representation of the Clifford algebra $\mathcal{C}_{m-1}$,
the focal submanifold $M_+=f^{-1}(1)$ is expressed as
\begin{eqnarray}
M_+
&=&\left\{x\in S^{2l-1}(1)~|~\langle P_0x, x\rangle=\langle P_1x, x\rangle=\cdots=\langle P_mx, x\rangle=0\right\}\nonumber\\
&=& \left\{(\textbf{a}, \textbf{b}) \in \mathbb{R}^l\oplus\mathbb{R}^l~\vline~\begin{array}{ll}\langle \textbf{a}, \textbf{b} \rangle=0, \,\,|\textbf{a}|^2=|\textbf{b}|^2=\frac{1}{2},\\
\langle \textbf{a}, E_i\textbf{b} \rangle=0, ~i=1,\cdots, m-1\end{array}\right\}.\label{M+}
\end{eqnarray}
Observe that in certain cases, even though the isoparametric family of OT-FKM type does not exist, the topological space $M_+$ defined in the same expression as above exists but not connected. 

\begin{rem}\label{Omega}
The class $\Omega_{l,m}$ contains several kinds of Stiefel manifolds those we are familiar with and $V_3(\mathbb{O}^n)$ of James that is our concern.
More precisely,
\begin{itemize}
\item[(1)] For $m=1$, $l=n\delta(1)=n$, $\Omega_{n,1}\cong V_3(\mathbb{R}^n)$;
\item[(2)] For $m=2$, $l=n\delta(2)=2n$, $\Omega_{2n,2}\cong V_3(\mathbb{C}^n)$;
\item[(3)] For $m=4$, $l=n\delta(4)=4n$, if $\{E_1, E_2, E_3\}$ is in the definite case (for example, define them in (\ref{4 def})), then $\Omega_{4n,4}\cong V_3(\mathbb{H}^n)$;
\item[(4)] For $m=8$, $l=n\delta(8)=8n$, if $\{E_1, \cdots, E_7\}$ is in the definite case, then $\Omega_{8n,8}\cong V_3(\mathbb{O}^n)$ of James. In other words, $V_3(\mathbb{O}^n)$ is a smooth manifold of dimension $3(8n-9)$, which is listed as Corollary \ref{v3ok}. For clarity, we can define $\{E_1, \cdots, E_7\}$ similarly as those in (\ref{4 def}) by replacing $\mathrm{i, j, k}$ with $e_1,\cdots, e_7$, where $\{1,e_1,e_2,\cdots, e_7\}$ is the standard orthonormal basis of the octonions $\mathbb{O}$. 
It follows from the Artin theorem that $e_i(e_jz)=-e_j(e_iz)$ for $i,j=1,\cdots, 7$, $i\neq j$ and $z\in\mathbb{O}$.
Hence, $\{E_1, \cdots, E_7\}$ becomes a representation of the Clifford algebra.

Surprisingly, we find that the point $A=\frac{1}{\sqrt{2}}\small{\left(\begin{matrix}e_7 &~~e_3\\ e_2 &-e_6\\~1& ~~e_4\end{matrix}\right)}\in V_3(\mathbb{O}^2)$, thus $V_3(\mathbb{O}^2)$ is a smooth manifold of dimension $21$. This completely exceeds one's expectation, including James'. \end{itemize}
\end{rem}

\begin{rem}\label{Y42}
The space $V_k(\mathbb{O}^n)$  is worthy of in-depth study, as many strange phenomena will occur in this space. For example,
$A=\frac{1}{\sqrt{2}}\small{\left(\begin{array}{ll}e_7 &~~e_3\\ e_2 &-e_6\\e_1& -e_5\\ ~1& ~~e_4\end{array}\right)}\in V_4(\mathbb{O}^2)$, thus $V_4(\mathbb{O}^2)$ is non-empty. Obviously, this could not happen in the real, complex or quaternion case.
\end{rem}

As mentioned in Remark \ref{Omega} (4), $V_3(\mathbb{O}^n)\cong\Omega_{8n,8}$ when $\{E_1, \cdots, E_7\}$ is in the definite case. Then we construct a smooth map between smooth manifolds
\begin{eqnarray}\label{pi2}
\pi: \quad V_3(\mathbb{O}^n)\cong\Omega_{8n,8}&\longrightarrow& S^{8n-1}(1)=V_1(\mathbb{O}^n),\\
A=\left(\textbf{a}, \textbf{b}, \textbf{c}\right)&\mapsto&\textbf{c}\nonumber
\end{eqnarray}
and give another partial answer to Question (1) of James:
\begin{prop}\label{onto}
When $n\geqslant 3$, the map $\pi: V_3(\mathbb{O}^n)\rightarrow V_1(\mathbb{O}^n)$ is surjective, but not submersive. In particular, $\pi$ is not a smooth fibre bundle.
\end{prop}

As smooth submanifolds in $S^{3l-1}(\sqrt{3})$, we are interested in various geometric and topological properties of $\Omega_{l,m}$. For instance, we establish minimality and the existence of circle bundle structure for certain $\Omega_{l,m}$ in Proposition \ref{m=4} and Proposition \ref{indefinite 4}.
More generally, we achieve the following

\begin{thm}\label{path}
When $m=1$, $n\geqslant 4$, or $m\geqslant 2$, $n\geqslant 3$, $\Omega_{l, m}$ is a path-connected closed smooth manifold. In particular, $V_3(\mathbb{O}^n)$ $(n\geqslant 3)$ is a path-connected closed smooth manifold.
\end{thm}

Furthermore, for the triviality of the tangent bundle of the smooth manifold $\Omega_{l,m}$, we have

\begin{thm}\label{para}
For $m>1$, odd or $m\equiv 0~(\mathrm{mod}~4)$, $\Omega_{l,m}$ is parallelizable. In particular, $V_3(\mathbb{O}^n)$ is parallelizable.
\end{thm}

The present paper is organized as follows. In Section \ref{sec2}, we give a proof of Theorem \ref{fiber}. In Section \ref{sec3}, we focus on the topology and geometry of $\Omega_{l,m}$, and give a proof of Proposition \ref{onto}. In order to generalize Corollary \ref{v3ok}  and solve completely Question (2) of James,
we investigate in the last section the set of critical or regular points in $V_k(\mathbb{O}^n)$ of $F$ defined by (\ref{Fk}).

\section{Proof of Theorem \ref{fiber}}\label{sec2}
In this section, we give a proof to
\vspace{3mm}

\noindent
\textbf{Theorem \ref{fiber}}
\emph{\,\, The projection $\pi: V_{k+1}(\mathbb{O}^n)\rightarrow V_k(\mathbb{O}^n)$ $(n>k\geqslant 2)$ is not a fibration in the sense of Serre.}
\vspace{2mm}

\begin{proof}
At the first step, define $A=\frac{1}{\sqrt{2}}\left(\begin{matrix}e_7 &~~e_3\\ e_2 &-e_6\end{matrix}\right)\in M_{2\times 2}(\mathbb{O})$, which will play an important role in our proof.
Recalling the Cayley-Dickson construction of the product in $\mathbb{O}\cong \mathbb{H}\oplus\mathbb{H}$:
\begin{eqnarray*}
\mathbb{O}\times \mathbb{O}&\longrightarrow&\mathbb{O}\\
(a,b),~(c,d)&\mapsto&(a,b)(c,d) := (ac-\bar{d}b, ~ da+b\bar{c}),
\end{eqnarray*}
we choose them in $\mathbb{H}\oplus\mathbb{H}$ as follows
\begin{eqnarray}\label{basis}
&&~1=(1, 0),~~e_1=(\mathrm{i}, 0), ~~e_2=(\mathrm{j}, 0), ~~e_3=(\mathrm{k}, 0), \\
&&e_4=(0, 1), ~~e_5=(0, \mathrm{i}), ~~e_6=(0, \mathrm{j}), ~~e_7=(0, \mathrm{k}),\nonumber
\end{eqnarray}
and see that $e_7e_2=e_3e_6=e_5$, \,$e_7e_3=-e_2e_6=e_4$. Then a
direct computation leads to
$$A\overline{A}^t=I_2,\quad \overline{A}^tA=\left(\begin{matrix}~1& -e_4\\ ~e_4&~~1\,\end{matrix}\right),$$
which suggests that the following equations for $x,y\in \mathbb{O}$
\begin{equation}\label{A eq}
A\left(\begin{matrix}x\\ y\end{matrix} \right)=\left(\begin{matrix}0\\ 0\end{matrix} \right).
\end{equation}
might have a non-zero solution.
In fact, writing $x=(x_1,x_2)$, $y=(y_1, y_2)\in \mathbb{H}\oplus\mathbb{H}$, we deduce easily the relations
$$y_1=-\mathrm{j}\overline{x}_2\mathrm{j}=-\mathrm{k}\overline{x}_2\mathrm{k},  \quad
y_2=\mathrm{j}\overline{x}_1\mathrm{j}=\mathrm{k}\overline{x}_1\mathrm{k}, $$
which implies $\mathrm{i}x_1=x_1\mathrm{i}$, $\mathrm{i}x_2=x_2\mathrm{i}$, thus $x_1=-y_2:=\alpha\in \mathbb{C}$, $x_2=y_1:=\beta\in\mathbb{C}$. Therefore, the solutions to the equations (\ref{A eq}) can be expressed as
$$\left(\begin{matrix}x\\ y\end{matrix} \right)=\left(\begin{matrix}(\alpha, ~~\beta)\\ (\beta, -\alpha)\end{matrix} \right).$$

Next, recalling the definition of $V_k(\mathbb{O}^n)$ by James in (\ref{james}),
we rewrite here the projection $\pi$ for $n>k\geqslant 2$ as follows
\begin{eqnarray*}
\pi: ~~~V_{k+1}(\mathbb{O}^n)~~~&\longrightarrow&~~~ V_k(\mathbb{O}^n)\\
\left(\begin{matrix}a_1\\ ~\vdots\\ a_{k+1}\end{matrix} \right)&\mapsto& \left(\begin{matrix}a_{2}\\~\vdots\\ a_{k+1}\end{matrix} \right).
\end{eqnarray*}

Choose a point $z_0=
(I_k, \textbf{0}_{k\times (n-k)})\in V_k(\mathbb{O}^n)$. It is not hard to see
$$\pi^{-1}(z_0)= \left\{\left(\begin{array}{c|c} \textbf{0}_{1\times k}& w_{1\times(n-k)}\\
\hline
I_k&\textbf{0}_{k\times (n-k)}\end{array} \right) \vline~ |w|=1\right\}.$$
That is,
$\pi^{-1}(z_0)\cong V_1(\mathbb{O}^{n-k})\cong S^{8(n-k)-1}$.

On the other hand, choose another point $z_0'=\begin{pmatrix}\left(\begin{matrix} A&\\
&I_{k-2} \end{matrix}\right), ~~\textbf{0}_{k\times (n-k)}\end{pmatrix}\in V_k(\mathbb{O}^n)$. We find
$$\pi^{-1}(z_0')=\left\{\left(\begin{array}{c|c} {\begin{array}{ccccc}\overline{x} &\overline{y} &
0&\cdots &0
\end{array}}
&
z_{1\times(n-k)}\\
\hline
{\begin{array}{cc} A&\\
&\qquad I_{k-2} \end{array}}
&\textbf{0}_{k\times (n-k)}\end{array} \right)
\vline \begin{array}{ll}
|x|^2+|y|^2+|z|^2=1,\\
(x, y)^t ~\text{is~ a~ solution~ to} ~(\ref{A eq})
\end{array} \right\}.$$
Since the space of solutions to the equation (\ref{A eq}) is of dimension $4$ as we saw before, the fiber $\pi^{-1}(z_0')$ is clearly homeomorphic to $S^{8(n-k)+3}$.

Therefore, for the two distinct points $z_0, z_0'\in V_k(\mathbb{O}^n)$, the corresponding fibers $\pi^{-1}(z_0)$ and $\pi^{-1}(z_0')$ are not homotopy equivalent to each other.
At last, since $V_2(\mathbb{O}^2)$ is a path-connected manifold as an $S^7$-bundle over $S^{15}$ (cf. \cite{Ja58}), there is a path in $V_k(\mathbb{O}^n)$ connecting these two points $z_0, z_0'\in V_k(\mathbb{O}^n)$.

Consequently, we can draw the conclusion that the projection $\pi: V_{k+1}(\mathbb{O}^n)\rightarrow V_k(\mathbb{O}^n)$ $(n>k\geqslant 2)$ is not a fibration in the sense of Serre,
whether the base space $V_k(\mathbb{O}^n)$ is path-connected or not.
\end{proof}



\section{The topology and geometry of $\Omega_{l,m}$}\label{sec3}

To study the topology of $\Omega_{l,m}$, we need to recall some properties of the focal submanifold $M_+$ of the OT-FKM type. It is well known that the isoparametric family of OT-FKM type exists if and only if $l-m-1\geqslant 1$. And when it exists, the focal submanifold $M_+$ is connected.
However, when $l-m-1=0$, even though the isoparametric family of OT-FKM type does not exist, the topological space $M_+$ defined in the same expression as above exists, but it is not connected. In fact, $l-m-1=0$ if and only if $(l, m)=(2, 1)$, $(4, 3)$, or $(8, 7)$. In these three cases, $M_+$ is $S^1\sqcup S^1$, $S^3\sqcup S^3$, or $S^7\sqcup S^7$, respectively, with two connected components (cf. \cite{FKM81}).

When $l-m-1\leqslant 0$, $\Omega_{l,m}$ is empty. Otherwise, it would follow from the definition (\ref{def omega}) that $\dim \textrm{Span}\{\textbf{a}, $ $E_1\textbf{a},\cdots, E_{m-1}\textbf{a}, \textbf{b}, \textbf{c}\}=m+2\geqslant l+1$, which contradicts the fact that $\dim \textrm{Span}\{\textbf{a}, $ $E_1\textbf{a},\cdots, E_{m-1}\textbf{a}, \textbf{b}, \textbf{c}\}\leqslant l$.

When $l-m-1=1$, \emph{i.e.}, $(l, m)=(3, 1)$, $(4, 2)$, or $(8, 6)$. In these three cases, $\Omega_{3,1}=O(3)$, $\Omega_{4, 2}$ and $\Omega_{8, 6}$ are empty. For clarity, we explain that for $\Omega_{4, 2}$. Observe that $\textrm{Span}\{\textbf{a}, $ $E_1\textbf{a}, \textbf{b}, \textbf{c}\}=\mathbb{R}^4$. But $E_1\textbf{b}\in \mathbb{R}^4$, evidently contradicts the definition of $\Omega_{4, 2}$. As for $\Omega_{8,6}$, one can also use a discussion by linear algebra.

When $l-m-1\geqslant m$, we see that $\Omega_{l,m}$ is non-empty. In this case, $l-m-1\geqslant m\geqslant 1$, the isoparametric family exists and there exist $\textbf{a}, \textbf{b}\in \mathbb{R}^l$ such that the point $\frac{1}{\sqrt{2}}(\textbf{a,b})$ lies in the focal submanifold $M_+$. Notice that the space $X=\textrm{Span}\{\textbf{a}, E_1\textbf{a}, ,\cdots, E_{m-1}\textbf{a}, \textbf{b}, E_1\textbf{b},\cdots, E_{m-1}\textbf{b}\}$ is a $2m$ dimensional subspace of $\mathbb{R}^l$. In the current case $2m\leqslant l-1<l$, there exists $\textbf{c}\in\mathbb{R}^l$ perpendicular to the space $X$, which reveals $\Omega_{l,m}\neq \varnothing$.

Now we are ready to give a proof of Theorem \ref{omega}.

\vspace{3mm}
\noindent
\textbf{Theorem \ref{omega}.}\,\,\,\emph{When it is not empty, $\Omega_{l, m}\subset \mathbb{R}^{3l}$ is a closed regular submanifold of dimension $3(l-m-1)$ with trivial normal bundle.}
\vspace{2mm}

\begin{proof}
On $\mathbb{R}^{3l}$, we consider the following $3m+3$ functions
\begin{eqnarray}\label{3m+3}
&&\omega_1=|\textbf{a}|^2-1, \quad\, \omega_2=|\textbf{b}|^2-1, \quad\,\,\, \omega_3=|\textbf{c}|^2-1,\nonumber\\
&& f_0=\langle \textbf{a}, \textbf{b} \rangle, ~~\qquad g_0=\langle \textbf{b}, \textbf{c} \rangle,~~\,\qquad h_0=\langle \textbf{c}, \textbf{a} \rangle, \qquad\\
&& f_i=\langle \textbf{a}, E_i\textbf{b} \rangle, ~~\,\quad g_j=\langle \textbf{b}, E_j\textbf{c} \rangle,~~\quad h_k=\langle \textbf{c}, E_k\textbf{a} \rangle, \nonumber
\end{eqnarray}
where $i, j, k=1,\cdots, m-1$. A direct calculation yields the Euclidean gradient vectors as follows
\begin{eqnarray}\label{gradients}
\nabla \omega_1&=&(~ 2\textbf{a}, \qquad 0, \,\,~\qquad 0 ~),\nonumber\\
\nabla \omega_2&=&(~~\,0,\qquad 2\textbf{b}, ~\qquad 0~),\nonumber\\
\nabla \omega_3&=&(~~\,0, ~\qquad 0, \qquad 2\textbf{c}~),\nonumber\\
\nabla f_0&=&(~~\,\textbf{b}, ~\qquad \textbf{a}, ~\qquad 0~),\\
\nabla f_i&=&( E_i\textbf{b}, \quad -E_i\textbf{a}, ~\quad 0~ ), ~~\quad i=1,\cdots, m-1,\nonumber\\
\nabla g_0&=&(~~\,0, ~~\qquad \textbf{c}, ~\qquad \textbf{b}~),\nonumber\\
\nabla g_j&=&(~~\,0, ~\,~\quad E_j\textbf{c}, \,\,\, -E_j\textbf{b} ), \quad j=1,\cdots, m-1,\nonumber\\
\nabla h_0&=&(~~\,\textbf{c}, ~~\qquad 0, ~~\qquad \textbf{a}~),\nonumber\\
\nabla h_k&=&(-E_k\textbf{c}, \quad 0, ~~~\,\,\quad E_k\textbf{a} ), \quad k=1,\cdots, m-1.\nonumber
\end{eqnarray}
For convenience, arrange the above $3m+3$ functions in order and denote them by ${F_1},\cdots,{F_{3m+3}}$, respectively. Then we have the following smooth map
\begin{equation}\label{F}
F: \mathbb{R}^{3l}\longrightarrow \mathbb{R}^{3m+3},\\
\end{equation}
so that $\Omega_{l,m}=F^{-1}(0)$. We will show that $dF$ is surjective at any point $(\textbf{a}, \textbf{b}, \textbf{c}) \in F^{-1}(0)$, hence $\Omega_{l,m}$ is a smooth manifold of dimension $3(l-m-1)$.

Observing that $|\nabla F_i|^2=4$ or $2$ and $\langle \nabla F_i, \nabla F_j\rangle =0$, $\forall ~i=1,2,3,4,m+4, 2m+4$ and $j\neq i$, it suffices to show
$$\nabla f_1, \cdots, \nabla f_{m-1}, \nabla g_1, \cdots, \nabla g_{m-1}, \nabla h_1, \cdots, \nabla h_{m-1}$$
are linearly independent at any point $(\textbf{a}, \textbf{b}, \textbf{c})\in \Omega_{l,m}$.

Let $G=G_{(3m-3)\times (3m-3)}$ be the positive semidefinite Gram
matrix defined by the inner products of the vectors $\nabla f_1, \cdots, \nabla f_{m-1}, \nabla g_1, \cdots, \nabla g_{m-1}, \nabla h_1, \cdots, \nabla h_{m-1}$. Define $\Phi_{\textbf{ab}}$ to be an $(m-1)\times (m-1)$ matrix with elements $(\Phi_{\textbf{ab}})_{ij}=\langle E_i\textbf{a}, E_j \textbf{b}\rangle$. Notice that we have an important equality
$$(\Phi_{\textbf{ab}})_{ij}=-(\Phi_{\textbf{ba}})_{ij}$$
 on $\Omega_{l.m}$ whenever $i=j$ or $i \neq j$, using the properties of Clifford matrices $E_1,\cdots, E_{m-1}$. 
Thus, $\Phi_{\textbf{ab}}=-\Phi_{\textbf{ba}}$, and $G$ can be decomposed as
\begin{equation*}
G=\left(\begin{array}{lll}
\,\,~~2I& -\Phi_{\textbf{ac}}&-\Phi_{\textbf{bc}}\\
-\Phi_{\textbf{ca}}& ~\,~2I& -\Phi_{\textbf{ba}}\\
-\Phi_{\textbf{cb}}& -\Phi_{\textbf{ab}}&~\,~2I
\end{array}\right)=
\left(\begin{array}{lll}
I&&\\
&I&\\
&&I\end{array}\right)+
\left(\begin{array}{lll}
I& \Phi_{\textbf{ca}}&\Phi_{\textbf{cb}}\\
\Phi_{\textbf{ac}}& ~I& \Phi_{\textbf{ab}}\\
\Phi_{\textbf{bc}}& \Phi_{\textbf{ba}}&~~I
\end{array}\right),
\end{equation*}
Hence, $G$ is the sum of a positive definite matrix and a positive semidefinite Gram matrix which is defined by the inner products of the vectors
$E_1\textbf{c},\cdots, E_{m-1}\textbf{c}, E_1\textbf{a}, $ $\cdots, E_{m-1}\textbf{a}, E_1\textbf{b},\cdots, E_{m-1}\textbf{b}$ in $\mathbb{R}^l$.
Therefore, $G$ is positive definite, and thus $dF$ is surjective at every point $(\textbf{a}, \textbf{b}, \textbf{c})\in \Omega_{l,m}=F^{-1}(0)$, as required. In fact, $\Omega_{l,m}$ is a closed smooth submanifold in $S^{3l-1}(\sqrt{3})$ .
\end{proof}

Next, 
we want to study the geometric properties of $\Omega_{l,m}$.
For instance, it is natural to ask if these submanifolds are minimal or austere in $S^{3l-1}(\sqrt{3})$. By definition, a submanifold $M$ of a Riemannian manifold $\widetilde{M}$ is \emph{austere}, if for each point $x\in M$ and each normal vector $\xi\in T_x^{\perp}M$, the set of eigenvalues of the shape operator $A_{\xi}$ is invariant (concerning multiplicities) under multiplication by $-1$. Clearly, an austere submanifold is automatically a minimal submanifold. As the first step in this respect, we show

\begin{prop}\label{m=4}
When $m=4$ and $\{E_1, E_2, E_3\}$ is in the definite case, $\Omega_{4n,4}\cong V_3(\mathbb{H}^n)\subset S^{12n-1}(\sqrt{3})$ is a minimal submanifold of codimension $14$. However, it is not austere.
\end{prop}

\begin{proof}
When $m=4$, $l=n\delta(m)=4n$. For $\textbf{a}=(a_1, a_2, \cdots, a_n)\in\mathbb{H}^n$, we define $\{E_1, E_2, E_3\}$ in the definite case as follows
\begin{eqnarray}\label{4 def}
E_1(\textbf{a})=~\mathrm{i}\textbf{a}&:=&(\mathrm{i}a_1, \mathrm{i}a_2,\cdots,\mathrm{i}a_{n}),\nonumber\\
E_2(\textbf{a})=~\mathrm{j}\textbf{a}&:=&(\mathrm{j}a_1, \mathrm{j}a_2,\cdots, \mathrm{j}a_{n}),\\
E_3(\textbf{a})=~\mathrm{k}\textbf{a}&:=&(\mathrm{k}a_1, \mathrm{k}a_2,\cdots, \mathrm{k}a_{n}).\nonumber
\end{eqnarray}
Defining the same $\omega_1$, $\omega_2$, $\omega_3, f_i, g_j, h_k$ $(i, j, k=0,\cdots, 3)$ as those in (\ref{3m+3}), we find
$$(\Phi_{\textbf{ab}})_{ij}=\langle E_i\textbf{a}, E_j \textbf{b}\rangle=0,$$
thus
$\Phi_{\textbf{ab}}=0$. Similarly, $\Phi_{\textbf{bc}}=\Phi_{\textbf{ca}}=0$. Equivalently speaking, the gradient vectors
$\nabla\omega_1$, $\nabla\omega_2$, $\nabla\omega_3, \nabla f_i, \nabla g_j, \nabla h_k$ $(i, j, k=0,\cdots, 3)$ are orthogonal to each other.

To show the minimality of $\Omega_{4n,4}$ in $S^{12n-1}(\sqrt{3})$, it suffices to prove that at any point
$x=(\textbf{a}, \textbf{b}, \textbf{c})\in \Omega_{4n,4}$,
$$\displaystyle \langle \sum_{i=1}^{12n-15}D_{e_i}e_i, \xi\rangle=0$$
for any normal vector $\xi$ of $\Omega_{4n,4}$ at $x$ in $S^{12n-1}(\sqrt{3})$, where $D$ is
the Levi-Civita connection of $\mathbb{R}^{12n}$, and $\{e_1, \cdots, e_{12n-15}\}$ is an orthonormal basis
of the tangent space of $\Omega_{4n,4}$ at $x$.
From the definition of $\Omega_{4n,4}$ and $\nabla f_i, \nabla g_j, \nabla h_k$ in (\ref{gradients}), we observe that $\frac{1}{\sqrt{2}}\nabla f_i, \frac{1}{\sqrt{2}}\nabla g_j, \frac{1}{\sqrt{2}}\nabla h_k$ $(i, j, k=0,\cdots, 3)$ are actually unit normal vector fields of $\Omega_{4n,4}$ in the sphere $S^{12n-1}(\sqrt{3})$, denoting them by $\xi_1,\cdots, \xi_{12}$ for short. To be specific, at a point $x=(\textbf{a}, \textbf{b}, \textbf{c})$, we represent $\frac{1}{\sqrt{2}}\nabla f_i$, $\frac{1}{\sqrt{2}}\nabla g_j$, $\frac{1}{\sqrt{2}}\nabla h_k$ as $xA_{i}^f$, $xA_j^g$ and $xA_k^h$ $(i, j, k=0,\cdots, 3)$, respectively, where
\begin{equation}\label{A}
\footnotesize{
A_0^f=\frac{1}{\sqrt{2}}\left(\begin{matrix}~~0&I_{4n}&0\\I_{4n}&0&0\\~~0&0&0\end{matrix}\right),~~
A_0^g=\frac{1}{\sqrt{2}}\left(\begin{matrix}0&0&0\\0&0&I_{4n}\\0&I_{4n}&0\end{matrix}\right),~~
A_0^h=\frac{1}{\sqrt{2}}\left(\begin{matrix}0&0&I_{4n}\\0&0&~~0\\I_{4n}&0&~~0\end{matrix}\right),}
\end{equation}
and
\begin{equation}\label{Ai}
\footnotesize{
A_i^f=\frac{1}{\sqrt{2}}\left(\begin{matrix}0&-\widetilde{D_i}&0\\\widetilde{D_i}&\,\,0&0\\0&\,\,0&0\end{matrix}\right),
A_i^g=\frac{1}{\sqrt{2}}\left(\begin{matrix}0&~~\,0&~~\,0\\0&~~\,0&-\widetilde{D_i}\\0&~~\widetilde{D_i}&~~\,\,0\end{matrix}\right),
A_i^h=\frac{1}{\sqrt{2}}\left(\begin{matrix}\,0&0&-\widetilde{D_i}\\\,0&0&\,\,0\\~~\widetilde{D_i}&0&\,\,0\end{matrix}\right),}
\end{equation}
where $\footnotesize{\widetilde{D_i}=\left(\begin{matrix}D_i&&\\&\ddots&\\&&D_i\end{matrix}\right)_{n\times n}}$  for $i=1,2,3$.
$D_1=  \left(\begin{smallmatrix}
0& 1 & & \\
-1& 0 & &\\
 & & 0 & 1 \\
 & & -1& 0
\end{smallmatrix}\right)
$,~ $D_2=  \left(\begin{smallmatrix}
 & & 1& 0  \\
 & & 0 & -1  \\
 -1& 0&  &  \\
 0&1 & &
\end{smallmatrix}\right)
$ and
$D_3=  \left(\begin{smallmatrix}
 & & 0& 1  \\
 & & 1 & 0  \\
 0& -1 &  &  \\
 -1& 0 & &
\end{smallmatrix}\right)$
are the skew-symmetric matrices corresponding to the left multiplication
by $\mathrm{i}, \mathrm{j}$ and $\mathrm{k}$ in $\mathbb{H}$.
Moreover, by virtue of $\nabla\omega_1, \nabla\omega_2, \nabla\omega_3$, we find that
$\xi_{13}:=\frac{1}{\sqrt{2}}(\textbf{a},0,-\textbf{c})$ and $\xi_{14}:=\frac{1}{\sqrt{6}}(\textbf{a},-2\textbf{b}, \textbf{c})$ are also unit normal vectors orthogonal to
$\xi_1,\cdots, \xi_{12}$. The matrices $A_{13}, A_{14}$ corresponding to $\xi_{13}, \xi_{14}$ are
\begin{equation}\label{1314}
A_{13}=\frac{1}{\sqrt{2}}\left(\begin{matrix}I_{4n}&&\\&0&\\&&-I_{4n}\end{matrix}\right),~~
A_{14}=\frac{1}{\sqrt{6}}\left(\begin{matrix}I_{4n}&&\\&-2I_{4n}&\\&&I_{4n}\end{matrix}\right).
\end{equation}

Let $A$ be one of the symmetric matrices in (\ref{A}), (\ref{Ai}) or (\ref{1314}). We see that at a position $x$,
\begin{eqnarray*}
\langle \sum_{i=1}^{12n-15}D_{e_i}e_i, xA\rangle&=&\sum_{i=1}^{12n-15}\langle -e_i, D_{e_i}(xA)\rangle=-\sum_{i=1}^{12n-15}\langle e_i, e_iA\rangle\\
&=&-\mathrm{Trace}A+\frac{1}{3}\langle xA, x\rangle+\sum_{\alpha=1}^{14}\langle \xi_{\alpha}A, \xi_{\alpha}\rangle.
\end{eqnarray*}
Clearly, we have $\mathrm{Trace} A=0$ and $\langle xA, x\rangle=0$. Moreover, one can check directly by the definition of $\Omega_{4n,4}$ that $\langle\xi_{\alpha}A_{\beta}, \xi_{\alpha}\rangle=0$ for any $\alpha=1,\cdots,14$ and $\beta=1,\cdots, 12$.
Besides, in virtue of a direct calculation, we find that it also holds for $A_{13}$ and $A_{14}$:
$$\sum_{\alpha=1}^{14}\langle\xi_{\alpha}A_{13}, \xi_{\alpha}\rangle= \sum_{\alpha=1}^{14}\langle\xi_{\alpha}A_{14}, \xi_{\alpha}\rangle=0.$$
Therefore, $\Omega_{4n,4}\subset S^{12n-1}(\sqrt{3})$ is a minimal submanifold, as demanded.

To complete the proof of the second part, it suffices to determine the eigenvalues of the shape operator at a certain point for some normal vector.
We first deal with the case $n=3$. In this case, $\Omega_{12,4}\cong V_3(\mathbb{H}^3)$ is a smooth manifold with $\dim\Omega_{12,4}=21$. Taking a point $x_0=(\textbf{a}, \textbf{b}, \textbf{c})\in\Omega_{12,4}$ with $\textbf{a}=(1, 0, 0)$, $\textbf{b}=(0,1,0)$ and $\textbf{c}=(0, 0, 1)\in\mathbb{H}^3$, for convenience, we represent $x_0$ alternatively by $x_0=\left(\begin{matrix}\textbf{a}\\\textbf{b}\\\textbf{c}\end{matrix}\right)=\left(\begin{matrix}1&0&0\\0&1&0\\0&0&1\end{matrix}\right)\in M_{3\times 3}(\mathbb{H})$. Recall that $\Big\{\frac{1}{2}\nabla\omega_1$, $\frac{1}{2}\nabla\omega_2$, $\frac{1}{2}\nabla\omega_3$, $\frac{1}{\sqrt{2}}\nabla f_i$, $\frac{1}{\sqrt{2}}\nabla g_j$, $\frac{1}{\sqrt{2}}\nabla h_k$ $(i, j, k=0, \cdots, 3)\Big\}$ constitutes a normal orthonormal basis of $\Omega_{12, 4}$ in $\mathbb{R}^{36}=M_{3\times 3}(\mathbb{H})$ at $x_0$. In the same way as representing $x_0$, we represent the normal basis vectors at $x_0$ alternatively, different from those in (\ref{gradients}), as follows
\begin{equation*}
\frac{1}{2}\nabla\omega_1=\left(\begin{matrix}1&0&0\\0&0&0\\0&0&0\end{matrix}\right),\quad
\frac{1}{2}\nabla\omega_2=\left(\begin{matrix}0&0&0\\0&1&0\\0&0&0\end{matrix}\right),\quad
\frac{1}{2}\nabla\omega_3=\left(\begin{matrix}0&0&0\\0&0&0\\0&0&1\end{matrix}\right),
\end{equation*}
\begin{equation*}
\frac{\nabla f_0}{\sqrt{2}}=\frac{1}{\sqrt{2}}\left(\begin{matrix}0&1&0\\1&0&0\\0&0&0\end{matrix}\right),~~\,
\frac{\nabla g_0}{\sqrt{2}}=\frac{1}{\sqrt{2}}\left(\begin{matrix}0&0&0\\0&0&1\\0&1&0\end{matrix}\right),~~\,
\frac{\nabla h_0}{\sqrt{2}}=\frac{1}{\sqrt{2}}\left(\begin{matrix}0&0&1\\0&0&0\\1&0&~~0\end{matrix}\right),
\end{equation*}

\begin{equation*}
\frac{\nabla f_i}{\sqrt{2}}=\frac{1}{\sqrt{2}}\left(\begin{matrix}0&-e_i&0\\ e_i&0&0\\0&0&0\end{matrix}\right),~~~
\frac{\nabla g_i}{\sqrt{2}}=\frac{1}{\sqrt{2}}\left(\begin{matrix}0&0&0\\0&0&-e_i\\0&e_i&0\end{matrix}\right),~~~
\frac{\nabla h_i}{\sqrt{2}}=\frac{1}{\sqrt{2}}\left(\begin{matrix}0&0&-e_i\\0&0&0\\e_i&0&0\end{matrix}\right)
\end{equation*}
for $i=1,2,3$ with $e_1=\mathrm{i}$, $e_2=\mathrm{j}$, $e_3=\mathrm{k}$. 
Therefore, we can identify the tangent space $T_{x_0}\Omega_{12, 4}$ with
$$T_{I_3}V_3(\mathbb{H}^3)=\{X_{3\times 3}\in M_{3\times 3}(\mathbb{H})~|~X+\overline{X}^t=0\}.$$
For $X=\begin{pmatrix} \lambda_1&y&z\\-\overline{y}&\lambda_2&w\\-\overline{z}&-\overline{w}&\lambda_3\end{pmatrix}\in T_{x_0}\Omega_{12, 4}$ with $\lambda_1, \lambda_2, \lambda_3\in\mathrm{Im}\mathbb{H}$, $y, z, w\in \mathbb{H}$, and for the normal vector $\xi_{14}=\frac{1}{\sqrt{6}}\begin{pmatrix}1&0&0\\0&-2&0\\0&0&1\end{pmatrix}
=\frac{1}{\sqrt{6}}\begin{pmatrix}1&0&0\\0&-2&0\\0&0&1\end{pmatrix}\,x_0
$ of $V_3(\mathbb{H}^3)$ in $S^{35}(\sqrt{3})$ at $x_0$,
the corresponding shape operator is given by
$$A_{\xi_{14}}X=-(D_X\xi_{14})^{\mathrm{T}}=-Y^{\mathrm{T}}=-\frac{1}{2}(Y-\overline{Y}^t),$$
where $(\cdot)^{\mathrm{T}}$ denotes the tangential projection and
$$Y:=D_{X}\xi_{14}
=\frac{1}{\sqrt{6}}\begin{pmatrix}1&0&0\\0&-2&0\\0&0&1\end{pmatrix}X
=\frac{1}{\sqrt{6}}\begin{pmatrix} \lambda_1&y&z\\2\overline{y}&-2\lambda_2&-2w\\-\overline{z}&-\overline{w}&\lambda_3\end{pmatrix}.$$
Hence, we obtain the principal curvatures: $\mu_1=-\frac{1}{\sqrt{6}}$ of multiplicity $10$, $\mu_2=\frac{1}{2\sqrt{6}}$ of multiplicity $8$, and $\mu_3=\frac{2}{\sqrt{6}}$ of multiplicity $3$, respectively. Evidently, $\Omega_{12, 4}$ is not austere.

In the case with $n\geqslant 4$, at the point $x_0=(\textbf{a}, \textbf{b}, \textbf{c})\in\Omega_{4n,4}$ with $\textbf{a}=(1,0,0,0,\cdots,0)$, $\textbf{b}=(0,1,0,0,\cdots,0)$ and $\textbf{c}=(0,0,1,0, \cdots,0)\in\mathbb{H}^n$, we can derive the same conclusion. 
\end{proof}

\begin{rem}
Actually, the minimality of $\Omega_{4n,4}\cong V_3(\mathbb{H}^n)$ in $S^{12n-1}(\sqrt{3})$ can be obtained as a corollary of the theory of Hsiang-Lawson \cite{HL71}. We present a direct proof in this paper to provide a potential way to study the extrinsic geometry of $\Omega_{4n,4}$ for all the indefinite $\{E_1, E_2, E_3\}$.\end{rem}

In contrast with the definite case, for the indefinite case with $m=4$, we obtain the following

\begin{prop}\label{indefinite 4}
When $m=4$, $l=8$ and $\{E_1, E_2, E_3\}$ is in the indefinite case, $\Omega_{8,4}$ is a circle bundle over $S^3\times S^3\times S^2$. In particular, it is connected.
\end{prop}

\begin{proof}
We first recall that when $m=4$, $l=8$, there is one definite isoparametric family of OT-FKM type, and one indefinite family.

For $\textbf{a}=(a_1, a_2)\in\mathbb{H}\oplus\mathbb{H}$, we define $\{E_1, E_2, E_3\}$ in the indefinite case as follows
\begin{equation*}
E_1(\textbf{a}):=(\mathrm{i}a_1, -\mathrm{i}a_2), ~~
E_2(\textbf{a}):=(\mathrm{j}a_1, -\mathrm{j}a_2),~~
E_3(\textbf{a}):=(\mathrm{k}a_1, -\mathrm{k}a_2).
\end{equation*}
Taking $\textbf{a}=\frac{1}{\sqrt{2}}(1, 1)$, $\textbf{b}=\frac{1}{\sqrt{2}}(\mathrm{i},\mathrm{i})$ and $\textbf{c}=\frac{1}{\sqrt{2}}(\mathrm{j},\mathrm{j})$,
we see easily $(\textbf{a}, \textbf{b}, \textbf{c})\in \Omega_{8,4}$. Thus it follows from Theorem \ref{omega} that $\Omega_{8,4}$ is a smooth manifold of dimension $9$.

Clearly, for any $(\textbf{a}, \textbf{b}, \textbf{c})=\big((a_1, a_2), (b_1, b_2), (c_1, c_2)\big)\in \Omega_{8,4}$, the condition $\langle \textbf{a}, \textbf{b} \rangle=\langle \textbf{a}, E_i\textbf{b} \rangle=0 ~~(i=1,2,3)$ is equivalent to $a_1\overline{b_1}+b_2\overline{a_2}=0$. Similarly, we can also obtain that $b_1\overline{c_1}+c_2\overline{b_2}=0$ and $c_1\overline{a_1}+a_2\overline{c_2}=0$. As a direct consequence, $a_1, a_2, b_1,b_2,c_1,c_2$ are all non-zero, and furthermore, $|a_1|=|a_2|=|b_1|=|b_2|=|c_1|=|c_2|=\frac{1}{\sqrt{2}}$.

Define a map $\phi$ from  $\Omega_{8,4}$ to the focal submanifold $M_+^{10}$ of OT-FKM type in the indefinite case with multiplicity $(4,3)$ as
\begin{eqnarray}
\phi: \quad \Omega_{8,4}&\longrightarrow& M_+^{10}\label{omega84}\\
(\textbf{a}, \textbf{b}, \textbf{c})&\mapsto& \frac{1}{\sqrt{2}}(\textbf{a}, \textbf{b}).\nonumber
\end{eqnarray}
Clearly, $\phi$ is not surjective. Moreover, we can show

\begin{lem}\label{map omega to M+}
$\frac{1}{\sqrt{2}}(\textbf{a}, \textbf{b})\in \mathrm{Image}(\phi)$ if and only if $|a_1|=|a_2|=\frac{1}{\sqrt{2}}$ and there exists $\xi\in \mathbb{H}$ with $|\xi|=1$ and $\mathrm{Re}~\xi=0$, such that $\textbf{a}=\xi\textbf{b}$.
\end{lem}

\noindent\emph{Proof of Lemma \ref{map omega to M+}:}
Since one can check the sufficient part directly, we will only consider the necessary part. Assume $\frac{1}{\sqrt{2}}(\textbf{a}, \textbf{b})\in \mathrm{Image}(\phi)$.
Then there exists $\textbf{c}\in \mathbb{H}\oplus\mathbb{H}$ such that $(\textbf{a}, \textbf{b}, \textbf{c})\in \Omega_{8,4}$. Because $\textbf{c}$ is perpendicular to the space spanned by $\{\textbf{a}, \textbf{b}, E_1\textbf{a}, E_2\textbf{a}, E_3\textbf{a}, E_1\textbf{b}, E_2\textbf{b}, E_3\textbf{b}\}$, the existence of $\textbf{c}$ is equivalent to the inequality
$$\dim \textrm{Span}\{\textbf{a}, \textbf{b}, E_1\textbf{a}, E_2\textbf{a}, E_3\textbf{a}, E_1\textbf{b}, E_2\textbf{b}, E_3\textbf{b}\}<8,$$
that is,
$$\dim \textrm{Span}\{E_1\textbf{a}, E_2\textbf{a}, E_3\textbf{a}, E_1\textbf{b}, E_2\textbf{b}, E_3\textbf{b}\}<6.$$
Equivalently, the Gram matrix $G$ defined by the inner products of the vectors $E_1\textbf{a}, E_2\textbf{a}, E_3\textbf{a}, E_1\textbf{b}, E_2\textbf{b}, E_3\textbf{b}$ expressed as
$$G=\left(\begin{matrix} ~~I&\Phi_{\textbf{ab}}\\
\Phi_{\textbf{ba}}&~~I\end{matrix}\right)$$
has determinant $0$, where
$$\Phi_{\textbf{ab}}=\left(\begin{matrix} 0&x&y\\-x&0&z\\-y&-z&0\end{matrix}\right)$$
with $x=\langle E_1\textbf{a}, E_2\textbf{b} \rangle$, $y=\langle E_1\textbf{a}, E_3\textbf{b} \rangle$, and $z=\langle E_2\textbf{a}, E_3\textbf{b} \rangle$.
A direct calculation leads to
$$0=\det G=\det(I+\Phi_{\textbf{ab}}^2)=\left( 1-(x^2+y^2+z^2)\right)^2.$$
That is to say, $$x^2+y^2+z^2=\langle a_1\overline{b_1}+a_2\overline{b_2}, \mathrm{i}\rangle^2+\langle a_1\overline{b_1}+a_2\overline{b_2}, \mathrm{j}\rangle^2+\langle a_1\overline{b_1}+a_2\overline{b_2}, \mathrm{k}\rangle^2=1.$$
Combining with the equality $a_1\overline{b_1}+b_2\overline{a_2}=0$,
which implies in particular that $\mathrm{Re}(a_1\overline{b_1}+a_2\overline{b_2})=0$, we obtain
$$|2\mathrm{Im}(a_1\overline{b_1})|=|a_1\overline{b_1}-\overline{a_1\overline{b_1}}|=|a_1\overline{b_1}+a_2\overline{b_2}|=1=2|a_1\overline{b_1}|,$$
where the last equality follows from $|a_1|=|b_1|=\frac{1}{\sqrt{2}}$.
Hence, $\mathrm{Re}(a_1\overline{b_1})=0$.
Define $\xi:=2a_1\overline{b_1}$. Thus $2a_2\overline{b_2}=-2\overline{a_1\overline{b_1}}=-\overline{\xi}=\xi$. Hence we get $(a_1, a_2)=\xi(b_1, b_2)$, \emph{i.e.}, $\textbf{a}=\xi\textbf{b}$. Now the proof of Lemma \ref{map omega to M+} is complete.

With Lemma \ref{map omega to M+} in mind, we observe without much difficulty that
$$\mathrm{Span}\{\textbf{a}, E_1\textbf{a}, E_2\textbf{a}, E_3\textbf{a}, \textbf{b}, E_1\textbf{b}, E_2\textbf{b}, E_3\textbf{b}\}=\mathrm{Span}\{\textbf{a}, E_1\textbf{a}, E_2\textbf{a}, E_3\textbf{a}, \textbf{b}, (a_1, -a_2)\},$$
and they are both of dimension $6$.

Therefore, we can define the following map
\begin{eqnarray*}
p: \quad\Omega_{8,4}&\longrightarrow& S^3(\frac{1}{\sqrt{2}})\times S^3(\frac{1}{\sqrt{2}})\times S^2(1)\\
(\textbf{a}, \textbf{b}, \textbf{c})&\mapsto& (a_1, a_2,\xi).
\end{eqnarray*}
We aim to prove that $p$ is a submersion. First of all, Lemma \ref{map omega to M+} reveals that the map $p$ is onto. It is sufficient to show that for any $(\textbf{a}^0, \textbf{b}^0, \textbf{c}^0)\in \Omega_{8,4}$, and
$$(\textbf{a}^0, \xi^0):=p(\textbf{a}^0, \textbf{b}^0, \textbf{c}^0)\in S^3(\frac{1}{\sqrt{2}})\times S^3(\frac{1}{\sqrt{2}})\times S^2(1),$$
the tangent map
$$dp: T_{(\textbf{a}^0, \textbf{b}^0, \textbf{c}^0)}\Omega_{8,4}\rightarrow T_{(\textbf{a}^0, \xi^0)}S^3(\frac{1}{\sqrt{2}})\times S^3(\frac{1}{\sqrt{2}})\times S^2(1)$$
is surjective. Given a tangent vector
$$X\in T_{(\textbf{a}^0, \xi^0)}S^3(\frac{1}{\sqrt{2}})\times S^3(\frac{1}{\sqrt{2}})\times S^2(1),$$
there exists a curve
$$\gamma: (-\varepsilon, \varepsilon)\rightarrow S^3(\frac{1}{\sqrt{2}})\times S^3(\frac{1}{\sqrt{2}})\times S^2(1),$$
defined by $\gamma(t)=(\textbf{a}(t), \xi(t))$ satisfying $\gamma(0)=(\textbf{a}^0, \xi^0)$ and $\gamma'(0)=X$. Define $\textbf{b}(t):=-\xi(t)\textbf{a}(t)$.
The curve $\gamma$ induces the following curve in the Stiefel manifold
$$\alpha:(-\varepsilon, \varepsilon)\rightarrow V_6(\mathbb{R}^8)$$
with $\alpha(t)=(\textbf{a}(t), E_1\textbf{a}(t), E_2\textbf{a}(t), E_3\textbf{a}(t), \textbf{b}(t), (a_1(t), -a_2(t)))$. Since the natural projection from $V_7(\mathbb{R}^8)$ to $V_6(\mathbb{R}^8)$ by dropping the last vector is a fiber bundle, one can lift the curve $\alpha(t)$ to a curve $\overline{\alpha}(t)$ in $V_7(\mathbb{R}^8)$ with
$$\overline{\alpha}(0)=(\textbf{a}^0, E_1\textbf{a}^0, E_2\textbf{a}^0, E_3\textbf{a}^0, \textbf{b}^0, (a_1^0, -a_2^0), \textbf{c}^0).$$
Define $\textbf{c}(t)$ to be the last component of $\overline{\alpha}(t)\in V_7(\mathbb{R}^8)$. Now, we obtain a curve
$$\overline{\gamma}:(-\varepsilon, \varepsilon)\rightarrow \Omega_{8,4}$$
with $\overline{\gamma}(0)=(\textbf{a}^0, \textbf{b}^0, \textbf{c}^0)$ and $p(\overline{\gamma}(t))=\gamma(t)$. Then, it is valid that $p$ is a surjective submersion with compact connected fibers. Hence, $p$ is a fiber bundle with fiber $S^1$. As a consequence, $\Omega_{8,4}$ is connected.

Now the proof of Proposition \ref{indefinite 4} is complete.
\end{proof}

Next, we will continue to study the connectedness for general $\Omega_{l,m}$. To achieve this goal, we need to prepare several lemmas.
Let $\{E_1, \cdots, E_{m-1}\}$ be a representation of the Clifford algebra $\mathcal{C}_{m-1}$ on $\mathbb{R}^l$ with $l=n\delta(m)$. 
When $m\not\equiv0~(\mathrm{mod}~4)$, let $\{f_1,\cdots, f_{m-1}\}$ be the unique irreducible representation of $\mathcal{C}_{m-1}$ on $\mathbb{R}^{\delta(m)}$. Then up to algebraic equivalence, for $(x_1,\cdots, x_n)\in \mathbb{R}^l=\oplus_{n}\mathbb{R}^{\delta(m)}$, $E_i(x_1,\cdots, x_n)=(f_ix_1,\cdots, f_ix_n)$ $(i=1,\cdots, m-1)$.
When $m\equiv0~(\mathrm{mod}~4)$, let $\{f_1,\cdots, f_{m-1}\}$ and $\{f_1',\cdots, f_{m-1}'\}$ be the two irreducible representations of $\mathcal{C}_{m-1}$ on $\mathbb{R}^{\delta(m)}$. For instance, the two representations can be chosen in such a way that $f_i'=-f_i$ $(i=1,\cdots, m-1)$. Then up to algebraic equivalence, there exists a natural number $0\leqslant q\leqslant n$ such that $E_i(x_1,\cdots, x_n)=(f_ix_1,\cdots f_ix_q, f_i'x_{q+1},\cdots, f_{i}'x_n)$ for $(x_1,\cdots, x_n)\in \mathbb{R}^l=\oplus_{n}\mathbb{R}^{\delta(m)}$, and $i=1,\cdots, m-1$.

Define
$$W_{l,m}:=\left\{(\textbf{a}, \textbf{b})\in \mathbb{R}^l\oplus\mathbb{R}^l~\vline~\begin{array}{ll}
\langle \textbf{a}, \textbf{b} \rangle=0,\,\,|\textbf{a}|^2=|\textbf{b}|^2=1, \\
\langle \textbf{a}, E_i\textbf{b} \rangle=0,
\,\,~\textrm{for}~i=1,\cdots, m-1.\end{array}\right\}.$$
Evidently, $W_{l,m}$ is diffeomorphic to the focal submanifold $M_+$ of OT-FKM type in (\ref{M+}), thus a smooth connected manifold. As a result, one has a smooth map
\begin{eqnarray}\label{phi'}
\phi':~~ \Omega_{l,m}&\longrightarrow& W_{l,m}\\
 (\textbf{a}, \textbf{b}, \textbf{c})&\mapsto& (\textbf{a}, \textbf{b}).\nonumber
\end{eqnarray}
For $(\textbf{a}, \textbf{b})\in W_{l,m}$, denote by $g(\textbf{a}, \textbf{b})$ the Gram matrix determined by the $2m$ vectors $\{\textbf{a}, E_1\textbf{a}, \cdots, E_{m-1}\textbf{a}, \textbf{b}, E_1\textbf{b}, \cdots, E_{m-1}\textbf{b}\}$.

\begin{lem}\label{fibration}
Let $U:=\{(\textbf{a}, \textbf{b})\in W_{l,m}~|~\mathrm{det}\,g(\textbf{a}, \textbf{b})\neq 0\}$.
Then $\phi'|_{\phi'^{-1}(U)}: \phi'^{-1}(U)\rightarrow U$ is a smooth fiber bundle provided $l>2m$.
\end{lem}

\begin{proof}
It is well known that the natural projection $V_{2m+1}(\mathbb{R}^{l})\rightarrow V_{2m}(\mathbb{R}^{l})$ by dropping the last vector is a smooth fiber bundle. Hence,  for any curve $\alpha$ with $\{\alpha(t)\}\subset U\subset W_{l,m}$, given any $\overline{\alpha}(0)\in \phi'^{-1}(\alpha(0))$, there exists a lift curve $\overline{\alpha}(t)$ in $\phi'^{-1}(U)$ such that $\phi'(\overline{\alpha}(t))=\alpha(t)$. Therefore, $\phi'|_{\phi'^{-1}(U)}: \phi'^{-1}(U)\rightarrow U$ is a surjective submersion with compact fibers (a sphere of dimension $l-2m-1$), and thus a smooth fiber bundle. In particular, it is a two-fold covering when $l=2m+1$.
\end{proof}

\begin{lem}\label{connectedness}
Let $M^m$ be a regular submanifold of dimension $m$ in $\mathbb{R}^r$, $N^n$ be a smooth manifold of dimension $n$, and $f: N\rightarrow\mathbb{R}^r$ be a smooth map. Assume that $f(N)$ is closed in $\mathbb{R}^r$, $M$ is connected and $m-n\geqslant 2$. Then $M\setminus f(N)$ is connected.
\end{lem}

\begin{proof}
For each point $p\in M$, there exists a coordinate chart $(U_p, \varphi_p=(x_1,\cdots, x_r))$ with $p\in U_p\subset \mathbb{R}^r$, such that
$$V_p:=U_p\cap M=\{q\in U_p~|~ x_{m+1}(q)=\cdots=x_{r}(q)=0\}$$
is a neighborhood of $p$ in $M$.
Moreover, we can require that $V_p$ is a strongly convex ball  centered at $p$ in $M$ with respect to the metric induced from $\mathbb{R}^r$.

At the first step, we will show that $V_p\setminus f(N)$ is path-connected.
For any two points $q_1, q_2\in V_p\setminus f(N)$,  there exists a unique minimal geodesic  $\gamma: [0, 1]\rightarrow V_p$ connecting $q_1$ and $q_2$: $\gamma(0)=q_1$, $\gamma(1)=q_2$. Then we can define an open cone $C(q_1, q_2, \varepsilon)\subset V_p$ as follows:
a point $q \in C(q_1, q_2, \varepsilon)$ if and only if there exists a geodesic $\alpha:[0, 1]\rightarrow V_p$ satisfying $\alpha(0)=q_1$, $d(q_1, \alpha(1))=d(q_1, q_2)$, and $d(q_2, \alpha(1))<\varepsilon$, such that $q=\alpha(t)$
for some $t\in (0,1)$. It is clear that for a small enough $\varepsilon>0$, $C(q_1, q_2, \varepsilon)$ is a well-defined
 open cone with vertex $q_1$ in $V_p$. Furthermore, we define $S_{\varepsilon}(q_2)$ as follows:
A point $q\in S_{\varepsilon}(q_2)$ if and only if there exists a geodesic $\alpha:[0, 1]\rightarrow V_p$, such that $\alpha(0)=q_1$, $\alpha(1)=q$, $d(q_1, q)=d(q_1, q_2)$, and $d(q_2, q)<\varepsilon$.
Because $f(N)$ is closed, we can choose a small enough $\varepsilon>0$ such that $S_{\varepsilon}(q_2)\subset V_p\setminus f(N)$.
   Then for each geodesic $\alpha:[0, 1]\rightarrow V_p$ with $\alpha(0)=q_1$, we derive
   a smooth map $\rho: C(q_1, q_2, \varepsilon)\rightarrow S_{\varepsilon}(q_2)$ by mapping $\alpha(t)$ to $\alpha(1)$ for $t\in (0,1)$.

Let $\pi: U_p\rightarrow V_p$ be the natural projection corresponding to the coordinate projection under $\varphi_p$.
Then a smooth map $\pi\circ f: f^{-1}(U_p)\rightarrow V_p$ is obtained,
and thus another smooth map $\sigma:=\rho\circ(\pi\circ f): (\pi\circ f)^{-1}(C(q_1, q_2, \varepsilon))\rightarrow S_{\varepsilon}(q_2)$.
By assumption, $n<m-1$,
then Sard's theorem implies that $\sigma$ is not surjective, which yields the existence of a minimal geodesic $\alpha:[0, 1]\rightarrow V_p$ with $\alpha(0)=q_1$, $d(q_1, \alpha(1))=d(q_1, q_2)$ and $d(q_2, q)<\varepsilon$, and thus the existence of a path from $q_1$ to $q_2$ in $V_p\setminus f(N)$. Hence, $V_p\setminus f(N)$ is path-connected.

Next, for any two points $q_1, q_2\in M\setminus f(N)$, it follows from the connectedness of $M$ that there exists a path $\beta: [0, 1]\rightarrow M$ connecting them. Noticing that $\beta([0, 1])$ is compact, then there exists a finite open covering of $\beta([0, 1])$ by strongly convex balls which satisfy all the properties described above. Therefore, there exists another path $\beta':[0, 1]\rightarrow M\setminus f(N)$ with $\beta'(0)=q_1$ and $\beta'(1)=q_2$. Consequently, $M\setminus f(N)$ is path-connected.
\end{proof}

\begin{lem}\label{degenerate}
Let $m>1$. For $(\textbf{a}, \textbf{b})\in W_{l,m}$, $\mathrm{det}~g(\textbf{a}, \textbf{b})=0$ if and only if there exist
$\lambda=(\lambda_1,\cdots, \lambda_{m-1}), \mu=(\mu_1,\cdots, \mu_{m-1})\in \mathbb{R}^{m-1}$ with $|\lambda|=|\mu|=1,$
and $\textbf{c}\in \mathbb{R}^l$ with $|\textbf{c}|=1$, such that $(\textbf{a}, \textbf{b})=((\sum_{i=1}^{m-1}\lambda_iE_i)\textbf{c}, (\sum_{i=1}^{m-1}\mu_iE_i)\textbf{c})$.
\end{lem}

\begin{proof}
We only need to prove the necessity.
 Observe that for $(\textbf{a}, \textbf{b})\in W_{l,m}$, $\mathrm{det}~g(\textbf{a}, \textbf{b})=0$ if and only if the vectors
$$E_1\textbf{a},\cdots, E_{m-1}\textbf{a},\, E_1\textbf{b},\cdots, E_{m-1}\textbf{b}$$
are linearly dependent. Hence, for $(\textbf{a}, \textbf{b})\in W_{l,m}$ with $\mathrm{det}~g(\textbf{a}, \textbf{b})=0$, there exists
$\xi=(\xi_1,\cdots, \xi_{m-1}), \eta=(\eta_1,\cdots, \eta_{m-1})\in \mathbb{R}^{m-1}$ such that
$$\sum_{i=1}^{m-1}\xi_iE_i\textbf{a}+\sum_{j=1}^{m-1}\eta_jE_j\textbf{b}=0.$$
Due to the fact that $E_1, \cdots, E_{m-1}$ are skew-symmetric, one derives $|\xi|=|\eta|\neq 0$.
Moreover, we can assume $|\xi|=|\eta|=1$. Otherwise, replace them by $\xi/|\xi|$ and $\eta/|\eta|$. Denoting by $\textbf{c}:=\sum_{i=1}^{m-1}\xi_iE_i\textbf{a}$, $\lambda:=-\xi$ and $\mu=:\eta$,
it is direct to see that $(\textbf{a}, \textbf{b})=((\sum_{i=1}^{m-1}\lambda_iE_i)\textbf{c}, (\sum_{i=1}^{m-1}\mu_iE_i)\textbf{c})$.
Hence, the lemma follows.
\end{proof}

\begin{lem}\label{U}
For $n\geqslant 2$, the subset $U=\{(\textbf{a}, \textbf{b})\in W_{l,m}~|~\mathrm{det}~g(\textbf{a}, \textbf{b})\neq 0\}\subset W_{l,m}$ is open and dense in $W_{l,m}$. Moreover, for $n\geqslant 3$, $U$ is path-connected.
\end{lem}

\begin{proof}
Clearly, $U$ is open.
For any $(\textbf{a}, \textbf{b})\in W_{l,m}\setminus U$ with
$\textbf{a}=(a_1,\cdots, a_n)$, $\textbf{b}=(b_1,\cdots, b_n)\in\oplus_n\mathbb{R}^{\delta(m)}$, Lemma \ref{degenerate} implies $|a_i|=|b_i|$ $(i=1,\cdots, n)$.
Without loss of generality, assume $|a_1|=|b_1|\neq 0$ and $x:=|a_1|=|b_1|\in(0, 1]$.
For $x\in (0, 1)$, let $\alpha$ be a curve in $W_{l,m}$ defined by $\alpha(t):=(\textbf{a}(t), \textbf{b}(t))$ with
\begin{eqnarray*}
\textbf{a}(t)&=&(a_1, \cos{t}a_2,\cdots, \cos{t}a_n)/\sqrt{\cos^2{t}+\sin^2{t}(x^2)},\\
\textbf{b}(t)&=&(\cos{t}b_1, b_2,\cdots, b_n)/\sqrt{\cos^2{t}+\sin^2{t}(1-x^2)}.
\end{eqnarray*}
Then $\alpha(0)=(\textbf{a}, \textbf{b})\in W_{l,m}\setminus U$. However, for any $t\neq 0$ which is close to $0$, Lemma \ref{degenerate} implies that $\alpha(t)\in U$.
For $x=1$, it follows by considering the curve
$\alpha(t):=(\textbf{a}(t), \textbf{b}(t))$ defined by
\begin{eqnarray*}
\textbf{a}(t)&=&(\cos{t}a_1, \cos{t}\sin{t}a_1, 0, \cdots, 0)/\sqrt{\cos^2{t}+\cos^2{t}\sin^2{t}},\\
\textbf{b}(t)&=&(\cos^2{t}b_1, \sin{t}b_1, 0, \cdots, 0)/\sqrt{\cos^4{t}+\sin^2{t}}.
\end{eqnarray*}
Hence, $U$ is dense in $W_{l,m}$.

Next, we will prove the path-connectedness of $U$ for $n\geqslant 3$. Consider a smooth map $f: S^{m-2}\times S^{m-2}\times S^{l-1}\longrightarrow \mathbb{R}^l\oplus\mathbb{R}^l$ defined by
$$f((\lambda_1,\cdots, \lambda_{m-1}), (\mu_1,\cdots, \mu_{m-1}), \textbf{c}):=(\sum_{i=1}^{m-1}\lambda_iE_i)\textbf{c}, (\sum_{i=1}^{m-1}\mu_iE_i)\textbf{c}).$$
Lemma \ref{degenerate} indicates that $U=W_{l,m}\setminus f(S^{m-2}\times S^{m-2}\times S^{l-1})$. Denote $N:=S^{m-2}\times S^{m-2}\times S^{l-1}$ and $M:=W_{l,m}$. Then the assumption $n\geqslant 3$ leads to
$\mathrm{dim} M-\mathrm{dim} N=(2n\delta(m)-m-2)-(n\delta(m)+2m-5)=n\delta(m)-3m+3\geqslant 3$. Therefore, it follows from Lemma \ref{connectedness} that $U=M\setminus N$ is path-connected.
\end{proof}

Now, we are prepared to prove

\vspace{3mm}
\noindent
\textbf{Theorem \ref{path}.}
\emph{
For $m=1$, $n\geqslant 4$, or $m\geqslant 2$, $n\geqslant 3$, $\Omega_{l, m}$ is a path-connected closed smooth manifold. In particular, $V_3(\mathbb{O}^n)$ $(n\geqslant 3)$ is a path-connected closed smooth manifold.
}
\vspace{2mm}

\begin{proof}
At the beginning, we would like to give some illustration. When $n\geqslant3 $, it is obvious that $l-m-1\geqslant m$, and $\Omega_{l, m}$ is non-empty as we explained before, thus a smooth manifold. In particular, when $m=1$ and $n=3$,  $\Omega_{3,1}=O(3)$, which is not connected. Thus we assume $n\geqslant 4$ to avoid this case.

Firstly, it follows from Lemma \ref{fibration} that $\phi'|_{\phi'^{-1}(U)}: \phi'^{-1}(U)\rightarrow U$ is a smooth fiber bundle with a positive dimensional sphere as the fiber. Thus, $\phi'^{-1}(U)$ is connected, since the base space $U$ is path-connected by Lemma \ref{U}.

Next, for any $(\textbf{a}, \textbf{b})\in W_{l,m}\setminus U$ with
$\textbf{a}=(a_1,\cdots, a_n)$ and $\textbf{b}=(b_1,\cdots, b_n)\in\oplus_n\mathbb{R}^{\delta(m)}$, Lemma \ref{degenerate} implies $|a_i|=|b_i|$ $(i=1,\cdots, n)$.  Without loss of generality,
assume $x:=|a_1|=|b_1|\in (0, 1]$.  We first consider the case $x\in (0, 1)$. For $\widetilde{\textbf{a}}=(0, a_2, \cdots, a_n)$, $\widetilde{\textbf{b}}=(0, b_2, \cdots, b_n)$,
provided $n\geqslant 3$, we have
$$\mathrm{dim}~\mathrm{Span}\{\widetilde{\textbf{a}}, E_1\widetilde{\textbf{a}},\cdots, E_{m-1}\widetilde{\textbf{a}}, \widetilde{\textbf{b}}, E_1\widetilde{\textbf{b}},\cdots, E_{m-1}\widetilde{\textbf{b}}\}< 2m\leqslant (n-1)\delta(m),$$
which indicates the existence of  $\widetilde{\textbf{c}}=(0, c_2,\cdots, c_n)\in\mathbb{R}^l$ with $|\widetilde{\textbf{c}}|=1$ such that
$$\widetilde{\textbf{c}}\perp \mathrm{Span}\{\widetilde{\textbf{a}}, E_1\widetilde{\textbf{a}},\cdots, E_{m-1}\widetilde{\textbf{a}}, \widetilde{\textbf{b}}, E_1\widetilde{\textbf{b}},\cdots, E_{m-1}\widetilde{\textbf{b}}\}.$$
Then one obtains a curve $\gamma$ in $\Omega_{l,m}$ given by $\gamma(t)=(\textbf{a}(t), \textbf{b}(t), \textbf{c}(t))$ with
\begin{eqnarray*}
\textbf{a}(t)&=&(a_1, \cos{t}a_2,\cdots, \cos{t}a_n)/\sqrt{\cos^2{t}+\sin^2{t}x^2},\\
\textbf{b}(t)&=&(\cos{t}b_1, b_2,\cdots, b_n)/\sqrt{\cos^2{t}+\sin^2{t}(1-x^2)},\\
\textbf{c}(t)&=&\widetilde{\textbf{c}},
\end{eqnarray*}
such that $\phi'(\gamma(0))\in W_{l,m}\setminus U$, and
$\phi'(\gamma(t))\in U$ for any $t\neq 0$ near $0$ by Lemma \ref{degenerate}.
For the case $x=1$, as in the proof of Lemma \ref{U}, the existence of the curve $\gamma$ follows from a similar and modified argument.
It is clear that $\phi'^{-1}((\textbf{a}, \textbf{b}))$ is a positive dimensional sphere and path-connected.
Consequently, each point in $\Omega_{l,m}$ can be connected to some point in $\phi'^{-1}(U)$ which is path-connected.
Therefore, $\Omega_{l,m}$ is path-connected as demanded.
\end{proof}

Now, concerning the parallelizability of $\Omega_{l,m}$, we will prove

\vspace{3mm}
\noindent
\textbf{Theorem \ref{para}}
\emph{
For $m>1$, odd or $m=0~(\mathrm{mod}~4)$, $\Omega_{l,m}$ is parallelizable. In particular, $V_3(\mathbb{O}^n)$ is parallelizable.
}
\vspace{2mm}

\begin{proof}
Recalling Theorem \ref{omega}, we see that $\Omega_{l, m}\subset \mathbb{R}^{3l}$ is a closed regular submanifold of dimension $3(l-m-1)$ with trivial normal bundle, which immediately implies that $\Omega_{l, m}$ is stably parallelizable.

To study the parallelizability of $\Omega_{l, m}$, the following observation is useful. Identify $\mathbb{R}^l\oplus\mathbb{R}^l\oplus\mathbb{R}^l$ with $M_{3\times l}(\mathbb{R})$. Consequently, a natural free smooth $\mathrm{SO}(3)$ action on $\Omega_{l,m}$ is induced by matrix multiplication. It implies that the tangent bundle of $\Omega_{l,m}$ admits a trivial sub-bundle of rank $3$, which corresponds to the tangent spaces of the orbits for the free $\mathrm{SO}(3)$ action.

For the case $m>1$ and $m$ is odd, $\mathrm{dim}~\Omega_{l,m}=3(l-m-1)$ is even. Since the tangent bundle of $\Omega_{l,m}$ admits a trivial sub-bundle of rank $3$, one infers that the Euler characteristic $\chi(\Omega_{l,m})$ is zero by Hopf's theorem. Moreover, by a theorem of Kervaire and Adams (cf. \cite{Sut64}), $\Omega_{l,m}$ is parallelizable.

For the case $m\equiv 0~(\mathrm{mod}~4)$, $\mathrm{dim}~\Omega_{l,m}=3(l-m-1)=3(n\delta(m)-m-1)\equiv 1~(\mathrm{mod}~4)$. Denote $\dim\Omega_{l,m}:=4p+1$ with $p\geqslant 1$.
As we have seen that $\Omega_{l, m}$ is stably parallelizable, $\Omega_{l, m}$ is parallelizable if and only if
$\hat{\chi}_2(\Omega_{l, m})\equiv 0~(\mathrm{mod}~2)$ by the theorem of Kervaire and Adams, where
$\hat{\chi}_2(\Omega_{l, m})$ is the $\mathrm{mod}~2$ semicharacteristic introduced by Kervaire (cf. \cite{Sut64}).
To finish the proof, we need only to show $\hat{\chi}_2(\Omega_{l, m})\equiv 0~(\mathrm{mod}~2)$. By a formula of Lusztig-Milnor-Peterson \cite{LMP69} (see also P.646 of \cite{Th69}), one has
$$\kappa(\Omega_{l, m})-\hat{\chi}_2(\Omega_{l, m})=w_2(\Omega_{l, m})\cup w_{4p-1}(\Omega_{l, m}),$$
where $\kappa(\Omega_{l, m})$ is the real Kervaire semicharacteristic (cf. \cite{Th69}), and $w_2(\Omega_{l, m})$, $w_{4p-1}(\Omega_{l, m})$
are the Stiefel-Whitney classes. From the fact that $\Omega_{l, m}$ is stably parallelizable, it follows $\kappa(\Omega_{l, m})-\hat{\chi}_2(\Omega_{l, m})\equiv 0~(\mathrm{mod}~2)$. On the other hand, using the fact that $\Omega_{l,m}$ admits a trivial sub-bundle of rank no less than $2$ and a theorem of Atiyah (see P. 647 of \cite{Th69}), we get $\kappa(\Omega_{l, m})=0$ as a $\mathrm{mod}~2$ integer. Hence, $\hat{\chi}_2(\Omega_{l, m})=0$ as a $\mathrm{mod}~2$ integer and thus $\Omega_{l, m}$ is parallelizable.
\end{proof}

At the end of this section, we would like to give a proof of Proposition \ref{onto},
which is promised in the end of Section \ref{sec1}.


\vspace{3mm}
\noindent
\textbf{Proposition \ref{onto}.}
\emph{
When $n\geqslant 3$, the map $\pi: V_3(\mathbb{O}^n)\cong\Omega_{8n,8}\rightarrow S^{8n-1}(1)=V_1(\mathbb{O}^n)$ defined in (\ref{pi2}) is surjective, but not submersive. In particular, $\pi$ is not a smooth fibre bundle.
}
\vspace{2mm}

\begin{proof}
Let $p$ be the projection
\begin{eqnarray*}
p: ~~V_2(\mathbb{O}^{n-1})&\longrightarrow& S^{8n-9}(1)=V_1(\mathbb{O}^{n-1})\\
\begin{pmatrix}a_1&\cdots &a_{n-1}\\c_1&\cdots &c_{n-1}\end{pmatrix}&\mapsto& (c_1,\cdots,c_{n-1}),
\end{eqnarray*}
which is a surjective submersion, proved by James.

We first show that $\pi$ is surjective. For any $\textbf{c}=(c_1,\cdots, c_n)\in S^{8n-1}(1)$, define
$$\cos^2\theta:=|c_1|^2+\cdots+|c_{n-1}|^2,\quad \sin^2\theta:=|c_n|^2.$$
Obviously,  if $\cos\theta=0$, then $\textbf{c}=(0,\cdots, 0, c_n)$ and $\pi^{-1}(\textbf{c})\cong V_2(\mathbb{O}^{n-1})$.

If $\sin\theta=0$, \emph{i.e.}, $\textbf{c}=(c_1, \cdots, c_{n-1}, 0)$, we can choose $A=\begin{pmatrix}a_1&\cdots&a_{n-1}&0\\
0&\cdots&0&1\\c_1&\cdots& c_{n-1}&0\end{pmatrix}$, where $\begin{pmatrix}a_1&\cdots&a_{n-1}\\c_1&\cdots&c_{n-1}\end{pmatrix}\in p^{-1}(c_1, \cdots, c_{n-1})$. It is clear that $A\in V_3(\mathbb{O}^n)\cong\Omega_{8n,8}$ and $\pi(A)=\textbf{c}$.

If $\sin\theta\neq 0$ and $\cos\theta\neq 0$, for $\textbf{c}=(c_1, \cdots, c_{n-1}, c_n)$, we choose
$$\begin{pmatrix}a_1&\cdots&a_{n-1}\\\frac{1}{\cos\theta}c_1&\cdots& \frac{1}{\cos\theta}c_{n-1}\end{pmatrix}\in p^{-1}(\frac{1}{\cos\theta}c_1~\cdots~\frac{1}{\cos\theta}c_n),$$
and
$$A=\begin{pmatrix}a_1&\cdots&c_{n-1}&0\\
\lambda c_1&\cdots&\lambda c_{n-1}&\mu c_n\\c_1&\cdots&c_{n-1}&c_n\end{pmatrix}$$
with $\lambda=-\frac{\sin\theta}{\cos\theta}$, $\mu=\frac{\cos\theta}{\sin\theta}$.
Clearly, $A\in V_3(\mathbb{O}^n)\cong\Omega_{8n,8}$ and $\pi(A)=\textbf{c}$.

In conclusion, $\pi: \Omega_{8n,8}\rightarrow S^{8n-1}(1)$ is surjective.

Next, we will show that $\pi$ is not submersive. It suffices to find a point $A_0\in V_3(\mathbb{O}^n)=\Omega_{8n,8}$ with $\pi(A_0)=\textbf{c}_0$, and a certain $Y\in T_{\textbf{c}_0}S^{8n-1}(1)$ such that there exists no $X\in T_{A_0}V_3(\mathbb{O}^n)$ satisfying $d\pi_{A_0}(X)=Y$ under the tangent map $d\pi_{A_0}:~~T_{A_0}V_3(\mathbb{O}^n)\longrightarrow T_{\textbf{c}_0}S^{8n-1}(1)$.

Take $A_0=\left(\begin{matrix}\textbf{a}\\\textbf{b}\\\textbf{c}\end{matrix}\right)=\frac{1}{\sqrt{2}}\begin{pmatrix}e_2&-e_6&0&\cdots&0\\e_1&-e_5& 0&\cdots&0\\1&e_4&0&\cdots&0\end{pmatrix}\in V_3(\mathbb{O}^n)$ with $\pi(A_0)=\textbf{c}_0=\frac{1}{\sqrt{2}}(1,e_4,0,\cdots, 0)\in S^{8n-1}(1)$, 
and choose $Y=e_2\textbf{a}=\frac{1}{\sqrt{2}}(-1, e_4, 0,\cdots, 0)$. Clearly, $\langle Y, \textbf{c}_0 \rangle=0$, and thus $Y\in T_{\textbf{c}_0}S^{8n-1}(1)$. Suppose that there exists $X=\left(\begin{matrix}\textbf{x}_1\\\textbf{x}_2\\\textbf{x}_3\end{matrix}\right)\in T_{A_0}V_3(\mathbb{O}^n)$ with $\textbf{x}_i=(x_{i1},\cdots,x_{in})\in \mathbb{O}^n$, such that
$$d\pi_{A_0}(X)=\textbf{x}_3=Y=e_2\textbf{a}.$$
Observe that $X\in T_{A_0}V_3(\mathbb{O}^n)$ if and only if $X$ is orthogonal to the following normal vectors at $A_0$
$$\nabla \omega_1=\begin{pmatrix} 2\textbf{a}\\0\\0\end{pmatrix},\quad
\nabla \omega_2=\begin{pmatrix} 0\\2\textbf{b}\\0\end{pmatrix},\quad
\nabla \omega_3=\begin{pmatrix} 0\\0\\2\textbf{c}\end{pmatrix},\qquad $$
$$\qquad\qquad \nabla f_i=\begin{pmatrix} e_i\textbf{b}\\ -e_i\textbf{a}\\0\end{pmatrix},
\nabla g_i=\begin{pmatrix} 0\\e_i\textbf{c}\\ -e_i\textbf{b}\end{pmatrix},
\nabla h_i=\begin{pmatrix} -e_i\textbf{c}\\ 0\\e_i\textbf{a}\end{pmatrix},~~~i=0, 1,\cdots, 7.
$$
Therefore, $\textbf{x}_1, \textbf{x}_2, \textbf{x}_3$ satisfy the following equalities
\begin{eqnarray}
\langle \textbf{x}_1, \textbf{a} \rangle&=&\langle \textbf{x}_2, \textbf{b} \rangle=0 \nonumber\\
\langle \textbf{x}_1, e_i\textbf{b} \rangle&=&\langle \textbf{x}_2, e_i\textbf{a} \rangle, \quad i=0, 1,\cdots, 7,\label{2.}\\
\langle \textbf{x}_2, e_i\textbf{c} \rangle&=&\langle \textbf{x}_3, e_i\textbf{b} \rangle,\quad i=0, 1,\cdots, 7,\label{3.}\\
\langle \textbf{x}_1, e_i\textbf{c} \rangle&=&\langle \textbf{x}_3, e_i\textbf{a} \rangle,\quad i=0, 1,\cdots, 7.\label{4.}
\end{eqnarray}
Notice that
\begin{equation}\label{1.}e_3\textbf{b}=e_2\textbf{c}, \quad e_3\textbf{a}=-e_1\textbf{c}, \quad e_2\textbf{a}=e_1\textbf{b}.
\end{equation}
Thus it follows from (\ref{3.}) and (\ref{1.}) that
$$
\langle \textbf{x}_2, e_3\textbf{a}\rangle=\langle \textbf{x}_2, -e_1\textbf{c}\rangle=-\langle \textbf{x}_3, e_1\textbf{b}\rangle=-\langle \textbf{x}_3, e_2\textbf{a}\rangle=-|e_2\textbf{a}|^2.$$
On the other hand, following from (\ref{2.}), (\ref{4.}) and (\ref{1.}), we also have
$$\langle \textbf{x}_2, e_3\textbf{a}\rangle=\langle \textbf{x}_1, e_3\textbf{b}\rangle=\langle \textbf{x}_1, e_2\textbf{c}\rangle=\langle \textbf{x}_3, e_2\textbf{a}\rangle=|e_2\textbf{a}|^2.$$
Consequently, $|e_2\textbf{a}|^2=0$, which gives a contradiction and hence the proof.

\end{proof}


\section{Regular and critical points of $F$ in $V_k(\mathbb{O}^n)$ }\label{sec4}
In Section \ref{sec3}, we proved that $\Omega_{l,m}$ is a regular submanifold in $\mathbb{R}^{3l}$ by successfully realizing
$\Omega_{l,m}$ as inverse image of a certain regular value for natural defined map on $\mathbb{R}^{3l}$. As a consequence, $V_3(\mathbb{O}^n)$ is a smooth manifold. For the next step to attack Question (2) of James, an intriguing problem is to study whether the similar idea works for $V_k(\mathbb{O}^n)$ with $k\geqslant 4$. Unfortunately, the analogous argument in Section \ref{sec3} fails in this case.
In this section, under the condition $k\geqslant 4$, we will pay attention to the distribution of critical points and regular points of the canonical map $F$ (defined in (\ref{Fk}) below) in $V_k(\mathbb{O}^n)=F^{-1}(0)$, which might be useful to understand the topological structure of the space itself. Of course, by the Whitney theorem in real algebraic geometry, the algebraic set $V_k(\mathbb{O}^n)$
is always decomposed as a finite disjoint union of smooth manifolds.

Using the identification $V_k(\mathbb{O}^n)\cong \{A\in M_{k\times n}(\mathbb{O})~|~A\overline{A}^t=I_k\}$,
we again construct a smooth map
\begin{eqnarray}\label{Fk}
F: M_{k\times n}(\mathbb{O})\cong \mathbb{R}^{8kn}&\rightarrow& \mathfrak{h}(k, \mathbb{O})\cong \mathbb{R}^{k(4k-3)}, \\
A&\mapsto& A\overline{A}^t-I_k,\nonumber
\end{eqnarray}
where
$\mathfrak{h}(k, \mathbb{O}):=\{X\in M_{k\times k}(\mathbb{O})~|~X=\overline{X}^t\}$.
Denote by $\mathcal{C}(F)$ the set of all critical points of $F$ in $V_k(\mathbb{O}^n)=F^{-1}(0)$, and $\mathcal{R}(F)=V_k(\mathbb{O}^n)\backslash \mathcal{C}(F)$.

Let us introduce an $O(k)$-action $\psi$ and an $O(n)$-action $\varphi$ on $V_k(\mathbb{O}^n)$.
For any $S\in O(k)$, $T\in O(n)$ and $A\in V_k(\mathbb{O}^n)\subset M_{k\times n}(\mathbb{O})$, define
\begin{eqnarray*}
\psi: O(k)\times V_k(\mathbb{O}^n)&\longrightarrow& V_k(\mathbb{O}^n)\\
(~~S, \quad A~~)\quad &\mapsto&\psi(S, A):=SA,
\end{eqnarray*}
and
\begin{eqnarray*}
\varphi: O(n)\times V_k(\mathbb{O}^n)&\longrightarrow& V_k(\mathbb{O}^n)\\
(~~T, \quad A~~)\quad &\mapsto&\varphi(T, A):=AT^t.
\end{eqnarray*}
Clearly, $\psi$ and $\varphi$ are well-defined on $V_k(\mathbb{O}^n)$, respectively.

\begin{rem}
It is interesting that $\psi$ is a free $O(k)$-action, but $\varphi$ is not free in general.
For instance, the $O(4)$-action of $\varphi$ on $V_4(\mathbb{O}^4)$ is not free. In fact, at the point $A=\frac{1}{\sqrt{2}}\small{\left(\begin{matrix}e_7 &~~e_3&0&0\\ e_2 &-e_6&0&0\\e_1& -e_5&0&0\\ ~1& ~~e_4&0&0\end{matrix}\right)}\in V_4(\mathbb{O}^4)$, the isotropy subgroup at $A$ contains
the subgroup $\{I_2\}\times O(2)$ in $O(4)$, which is not trivial.
\end{rem}

The following result yields that both the sets $\mathcal{C}(F)$ and $\mathcal{R}(F)$ are not discrete. 
\begin{prop}\label{inv}
The sets $\mathcal{C}(F)$ and $\mathcal{R}(F)$ are both invariant under the $O(k)$-action $\psi$ and the $O(n)$-action $\varphi$ on $V_k(\mathbb{O}^n)$, respectively.
\end{prop}

\begin{proof}
We need only to consider $\mathcal{C}(F)$. Given $A\in \mathcal{C}(F)\subset M_{k\times n}(\mathbb{O})$, we will show $\psi(S, A), \varphi(T, A)\in \mathcal{C}(F)$ for any $S\in O(k)$ and $T\in O(n)$.

Observe that for $F: \mathbb{R}^{8kn}\cong M_{k\times n}(\mathbb{O})\rightarrow \mathbb{R}^{k(4k-3)}$,
$$F(SA)=(SA)\overline{(SA)}^t-I_k=S(F(A))S^t\in \mathfrak{h}(k, \mathbb{O})\cong \mathbb{R}^{4k^2-3k}.$$
Thus for $dF_A: M_{k\times n}(\mathbb{O})\rightarrow \mathfrak{h}(k, \mathbb{O})$ and $dF_{SA}: M_{k\times n}(\mathbb{O})\rightarrow \mathfrak{h}(k, \mathbb{O})$,
the dimensions satisfy
$\mathrm{dim~Ker}(dF_A)=\mathrm{dim~Ker}(dF_{SA})$. Consequently, $A\in \mathcal{C}(F)$ implies $SA \in \mathcal{C}(F)$, \emph{i.e.}, $\mathcal{C}(F)$ is invariant under the $O(k)$-action $\psi$.

Similarly, for the $O(n)$-action $\varphi$,
$$F(AT^t)=(AT^t)\overline{(AT^t)}^t-I_k=F(A)\in \mathfrak{h}(k, \mathbb{O})\cong \mathbb{R}^{4k^2-3k},$$
and $\mathrm{dim~Ker}(dF_A)=\mathrm{dim~Ker}(dF_{AT^t})$. Hence, $A\in \mathcal{C}(F)$ implies $AT^t \in \mathcal{C}(F)$, \emph{i.e.}, $\mathcal{C}(F)$ is invariant under the $O(n)$-action $\varphi$.
\end{proof}
\begin{rem}\label{ok}
For $n=k$, as a direct consequence of Proposition \ref{inv}, $O(k)\subset \mathcal{R}(F)$.
\end{rem}

For $A\in V_{k}(\mathbb{O}^n)$, \emph{i.e.}, $A\overline{A}^t=I_k$, define a real vector space
\begin{equation}\label{VA}
V_A:=\{X\in M_{k\times n}(\mathbb{O})~|~A\overline{X}^t+X\overline{A}^t=0\}.
\end{equation}
For $n=k=2$, it is known that $F^{-1}(0)=V_2(\mathbb{O}^2)\cong M_+^{22}$ is a regular submanifold of dimension $22$ in $M_{2\times 2}(\mathbb{O})\cong\mathbb{R}^{32}$ (cf. \cite{QTY22}). Actually, in this case one can prove directly that $V_A=T_AV_2(\mathbb{O}^2)$,
the tangent space of $V_2(\mathbb{O}^2)$ at $A$. Since $F$ consists of quadratic polynomials, for $A\in M_{k\times n}(\mathbb{O})$, the the following formula holds clearly
$$\mathrm{dim}V_A+\mathrm{rank}J_A=8kn,$$
where $J_A$ is the Jacobian matrix of $F$ at $A$.
Therefore, to seek regular or critical points of the smooth map $F$, it suffices to determine the dimension of the space $V_A$.


By the natural inclusions 
$\mathbb{C}\subset \mathbb{H}\subset \mathbb{O}$, one has
the induced inclusions $U(k)\subset Sp(k)\subset V_k(\mathbb{O}^k)$. Using the definition of $V_A$, we get

\begin{prop}\label{uk}
Any point $A$ in $U(k)\subset V_k(\mathbb{O}^k)$ belongs to $\mathcal{R}(F)$.
\end{prop}
\begin{proof}
For $A\in M_{k\times k}(\mathbb{O})$ satisfying $A\overline{A}^t=I_k$, the kernel of the tangent map $dF_A: M_{k\times k}(\mathbb{O}) \rightarrow\mathfrak{h}(k, \mathbb{O})$ at $A$ is $V_A$ in (\ref{VA}).

The equality $\mathrm{dim} V_{I_k}=\mathrm{dim} M_{k\times k}(\mathbb{O})-\mathrm{dim} \mathfrak{h}(k, \mathbb{O})$ implies $I_k \in \mathcal{R}(F)$. We will show for $A\in U(k)$, $\mathrm{dim} V_A=\mathrm{dim} V_{I_k}$. Define $L: V_A\rightarrow V_{I_k}$ by $L(X)=A\overline{X}^t$. Clearly, $L$ is a well-defined real linear transformation.
Now, for $A\in U(k)$ and $X\in V_A$, $A\overline{X}^t=0$ implies
$$\overline{X}^t=(\overline{A}^tA)\overline{X}^t=\overline{A}^t(A\overline{X}^t)=0,$$
where the second equality is guaranteed by the formula that $x(yz)=(xy)z$ if $x, y\in \mathbb{C}\subset \mathbb{H}\subset \mathbb{O}$ and $z\in\mathbb{O}$.
It follows that $L$ is injective. 
Therefore, the fact $\mathrm{dim} V_A\geqslant \mathrm{dim} M_{k\times k}(\mathbb{O})-\mathrm{dim} \mathfrak{h}(k, \mathbb{O})=\mathrm{dim} V_{I_k}$ implies that $L$ is a real linear isomorphism, and thus $\mathrm{dim} V_A=\mathrm{dim} V_{I_k}$.
\end{proof}


The following result shows that the set $\mathcal{C}(F)$ is not empty in $V_k(\mathbb{O}^n)$ for $k\geqslant 4$. This unsatisfactory phenomenon does not occur when $k=2,3$.

\begin{prop}
Assume $n\geqslant k\geqslant 4$. For the smooth map $F$ in (\ref{Fk}), 
there exists at least one critical point in $V_k(\mathbb{O}^n)$. In fact, $0\in \mathfrak{h}(k, \mathbb{O})$ is a critical value of $F$.
\end{prop}

\begin{proof}
We start with the case for $n=k=4$. It is known that $I_4$ is a regular point with $\dim V_{I_4}=76$.
Set
\begin{equation}\label{B4}
B_4:=\left(\begin{matrix}A&\\&A\end{matrix} \right)\,\, \textrm{with} \,\,A=\frac{1}{\sqrt{2}}\left( \begin{matrix}e_7&e_3\\1&e_4\end{matrix}\right)\in V_2(\mathbb{O}^2).
\end{equation}
Cleary $B_4\in V_4(\mathbb{O}^4)$.
For $X\in M_{4\times 4}(\mathbb{O})$,
write $\overline{X}^t:=\left( \begin{matrix}W_{2\times 2}&U_{2\times 2}\\ V_{2\times 2}&Z_{2\times 2}\end{matrix}\right)$.
By definition (\ref{VA}), $X\in V_{B_4}$
if and only if
$$AW+\overline{AW}^t=0,\quad AZ+\overline{AZ}^t=0,\quad AU+\overline{AV}^t=0.$$
It follows that $V_{B_4}$ has a decomposition $V_1\oplus V_2\oplus V_3$,
where $$V_1:=\{X\in V_{B_4}~|~AW+\overline{AW}^t=0, U=0, V=0, Z=0\},$$
$$V_2:=\{X\in V_{B_4}~|~AZ+\overline{AZ}^t=0, U=0, V=0, W=0\},$$
and
$$V_3:=\{X\in V_{B_4}~|~AU+\overline{AV}^t=0, Z=0, W=0\}.$$
Following from the discussion on $V_A$ for $A\in V_2(\mathbb{O}^2)$,
we know that $\mathrm{dim} V_1=\mathrm{dim} V_2=22$. To determine the dimension of $V_3$, it is useful to decompose
$V_3$ into $V_{31}\oplus V_{32}$, where
$$V_{31}:=\{X\in V_{B_4}~|~AU+\overline{AV}^t=0, U=V,Z=0, W=0\},$$
and
$$V_{32}:=\{X\in V_{B_4}~|~AU+\overline{AV}^t=0, U=-V,Z=0, W=0\}.$$
It is direct and easy to verify
$\mathrm{dim}~V_{31}=22$ and $\mathrm{dim}~V_{32}=14$.
Therefore, $\dim V_{B_4}= 22+22+22+14=80>76=\dim V_{I_4}$. Equivalently speaking, ${B_4}\in V_4(\mathbb{O}^4)$ is a critical point of $F$.


More generally, for $V_k(\mathbb{O}^n)$ with $n\geqslant k\geqslant 4$, it follows easily that $I_{k\times n}:=\left( I_k, \textbf{0}_{k\times (n-k)}\right)\in V_k(\mathbb{O}^n)$ is a regular point of $F$, and $\dim V_{I_{k\times n}}=8kn-4k^2+3k$. Moreover, choosing $B_{k\times n}=\begin{pmatrix}\left(\begin{array}{cc}B_4&\\&I_{k-4} \end{array}\right), ~~\textbf{0}_{k\times (n-k)}\end{pmatrix}\in V_k(\mathbb{O}^n)$,
one can verify that $\dim V_{B_{k\times n}}= 8kn-4k^2+3k+4>\dim V_{I_{k\times n}}$. It follows that $B_{k\times n}$ is a critical point of $F$.

\end{proof}

From now on, we will focus on the case of $V_4(\mathbb{O}^4)$. For convenience, write $A\in V_4(\mathbb{O}^4)$ as $A=\left(\begin{matrix}\textbf{a}\\ \textbf{b}\\ \textbf{c}\\ \textbf{d}\end{matrix}\right)=\left(\begin{matrix}a_{1}&a_{2}&a_{3}& a_4\\b_{1}&b_{2}&b_{3}& b_4\\c_{1}&c_{2}&c_{3}&c_4\\d_1&d_2&d_3&d_4\end{matrix}\right)\in M_{4\times 4}(\mathbb{O})$. Similarly as what we did in Section \ref{sec3}, the following real functions corresponding to the components of $F(A)=A\overline{A}^t-I_4$ will be considered
$$\omega_1=|\textbf{a}|^2-1, \quad \omega_2=|\textbf{b}|^2-1, \quad \omega_3=|\textbf{c}|^2-1, \quad \omega_4=|\textbf{d}|^2-1,$$
$$~f_0=\langle \textbf{a}, \textbf{b} \rangle, \,\,\,~~ g_0=\langle \textbf{a}, \textbf{c} \rangle,~~\,\,\, h_0=\langle \textbf{a}, \textbf{d} \rangle, ~~\,\,\, p_0=\langle \textbf{b}, \textbf{c} \rangle,~~\,\,\, q_0=\langle \textbf{b}, \textbf{d} \rangle,~~\,\,\, r_0=\langle \textbf{c}, \textbf{d}\rangle,\,$$
$$f_i=\langle \textbf{a}, e_i\textbf{b} \rangle, ~~ g_i=\langle \textbf{a}, e_i\textbf{c} \rangle,~~ h_i=\langle \textbf{a}, e_i\textbf{d} \rangle,~~ p_i=\langle \textbf{b}, e_i\textbf{c} \rangle,~~ q_i=\langle \textbf{b}, e_i\textbf{d} \rangle, ~~r_i=\langle \textbf{c}, e_i\textbf{d}\rangle, $$
where
$i=1,\cdots, 7$,  and
$e_i\textbf{a}:=(e_ia_1, e_ia_2, e_i a_3, e_ia_4)$, verbatim for $e_i\textbf{b}$, $e_i\textbf{c}$ and $e_i\textbf{d}$.

As in Theorem \ref{omega}, to check whether $dF_A$ is surjective, it suffices to check whether
$$\nabla f_1,\cdots, \nabla f_7,\cdots\cdots, \nabla r_1,\cdots, \nabla r_7$$
are linearly independent. It is equivalent to check whether the Gram matrix $G$ defined by inner products of $\nabla f_1,\cdots, \nabla f_7,\cdots\cdots, \nabla r_1,\cdots, \nabla r_7$ is positive definite. More precisely, $G$ can be directly calculated to be
\begin{equation}\label{G}
G=\begin{pmatrix}
2I_7& \Phi_{\textbf{bc}}& \Phi_{\textbf{bd}}&\Phi_{\textbf{ca}}&\Phi_{\textbf{da}}&0\\
 \Phi_{\textbf{cb}}&2I_7&\Phi_{\textbf{cd}}&\Phi_{\textbf{ab}}&0&\Phi_{\textbf{da}}\\
\Phi_{\textbf{db}}&\Phi_{\textbf{dc}}&2I_7&0&\Phi_{\textbf{ab}}&\Phi_{\textbf{ac}}\\
 \Phi_{\textbf{ac}}&\Phi_{\textbf{ba}}&0&2I_7&\Phi_{\textbf{cd}}&\Phi_{\textbf{db}}\\
\Phi_{\textbf{ad}}&0&\Phi_{\textbf{ba}}&\Phi_{\textbf{dc}}&2I_7&\Phi_{\textbf{bc}}\\
0&\Phi_{\textbf{ad}}&\Phi_{\textbf{ca}}&\Phi_{\textbf{bd}}&\Phi_{\textbf{cb}}&2I_7  \end{pmatrix},
\end{equation}
where $(\Phi_{\textbf{ab}})_{ij}=:\langle e_i\textbf{a}, e_j\textbf{b}\rangle=-(\Phi_{\textbf{ba}})_{ij}$, $i, j=1,\cdots, 7$.

From another point of view, observe that the gradients $\nabla f_i, \cdots, \nabla r_i$ $(i=1,\cdots, 7)$ are equal to
\begin{eqnarray*}
\nabla f_i&=&(e_i\textbf{b}, ~~-e_i\textbf{a}, \quad 0, \quad~~\,\,\,\,\, \,0\,\,\,),\\
\nabla g_i&=&(~e_i\textbf{c}, \quad 0, ~~\,\,-e_i\textbf{a},\,\,\,\,\, 0\,\,\, ),\\
\nabla h_i&=&(~e_i\textbf{d},\quad 0, \quad\,\,\, 0, \,\,\,-e_i\textbf{a}\,),\\
\nabla p_i&=&(\,\,\,\,0, \quad e_i\textbf{c}, \,\,-e_i\textbf{b}, \,\,\,\,\, 0~~\,\,),\\
\nabla q_i&=&(\,\,\,\,0, \quad e_i\textbf{d}, \quad\, 0, \,\,\,-e_i\textbf{b}\,),\\
\nabla r_i&=&(\,\,\,\,0,\quad \,\,\,0, \,\,\,\,\,\,e_i\textbf{d}, \,\,-e_i\textbf{c}\,).
\end{eqnarray*}
Thus $A\in V_4(\mathbb{O}^4)$ is a critical point of $F$ if and only if there exist $x_1,\cdots, x_6\in \mathrm{Im}\mathbb{O}$, which are not simultaneously zero, such that
\begin{eqnarray*}
x_1\textbf{b}+x_2\textbf{c}+x_3\textbf{d}&=&0,\\
-x_1\textbf{a}+x_4\textbf{c}+x_5\textbf{d}&=&0,\\
-x_2\textbf{a}-x_4\textbf{b}+x_6\textbf{d}&=&0,\\
-x_3\textbf{a}-x_5\textbf{b}-x_6\textbf{c}&=&0.
\end{eqnarray*}
In other words, $A\in V_4(\mathbb{O}^4)$ is a critical point of $F$ if and only if there exists a non-zero skew-symmetric matrix $\xi=\left(\xi_{ij}\right)_{4\times 4}$ with $\xi_{ij}\in\mathrm{Im}\mathbb{O}$, such that
\begin{equation}\label{xiA}
\xi A=0.
\end{equation}
As a matter of fact, each row of $\xi$ is non-zero. To see this, suppose that the first row of $\xi$ is zero, \emph{i.e.}, $\xi=\left(\begin{matrix} 0&0&0&0\\
0&0&x_4&x_5\\0&-x_4&0&x_6\\0&-x_5&-x_6&0\end{matrix}\right)$, and the point $A=\left(\begin{matrix} \textbf{a}\\\textbf{b}\\\textbf{c}\\\textbf{d}\end{matrix}\right)\in V_4(\mathbb{O}^4)$ is a critical point of $F$. Thus $B=\left(\begin{matrix} \textbf{b}\\\textbf{c}\\\textbf{d}\end{matrix}\right)\in V_3(\mathbb{O}^4)$ would be a critical point of the $F$ in (\ref{F}). However, as we have shown in Theorem \ref{omega} that $V_3(\mathbb{O}^4)=\Omega_{32, 8}$ is a smooth manifold, containing no critical point of $F$ in (\ref{F}). There comes a contradiction.

As an application of (\ref{xiA}), we give an alternative proof for the assertion that the point $B_4=\left(\begin{matrix}A&\\&A\end{matrix} \right)\in V_4(\mathbb{O}^4)$ with $A=\frac{1}{\sqrt{2}}\left( \begin{matrix}e_7&e_3\\1&e_4\end{matrix}\right)\in V_2(\mathbb{O}^2)$ in (\ref{B4}) is a critical point of $F$. In fact, choosing $x_1=x_6=0$, $x_2=-x_5=e_1$ and $x_3=x_4=-e_6$, we find that $\xi=\left(\begin{matrix}0&0&e_1&-e_6\\0&0&-e_6&-e_1\\-e_1&e_6&0&0\\e_6&e_1&0&0\end{matrix}\right)$ satisfies $\xi B_4=0$, since $e_6 e_3=-e_5=-e_1e_4$, and $e_1e_7=e_6$.
Moreover, we can prove similarly that the point $B'_4=\left(\begin{matrix}A'&\\&A'\end{matrix} \right)\in V_4(\mathbb{O}^4)$ with $A'=\frac{1}{\sqrt{2}}\left( \begin{matrix}e_2&-e_6\\e_1&-e_5\end{matrix}\right)\in V_2(\mathbb{O}^2)$
is also a critical point of $F$.

At the end of this section, we establish the following sufficient and necessary conditions for points of special type in $V_4(\mathbb{O}^4)$ to be critical points of $F$.

\begin{prop}\label{diag A44}
Given $A=\begin{pmatrix} \textbf{a}\\ \textbf{b}\end{pmatrix}$
$\in V_2(\mathbb{O}^2)$,
the point $B=\begin{pmatrix} A&0\\0&A \end{pmatrix}\in V_4(\mathbb{O}^4)$ is a critical point of $F$ if and only if $A\in \mathrm{Im}~\pi$, where $\pi:~~V_3(\mathbb{O}^2)\longrightarrow V_2(\mathbb{O}^2)$ is given by $\pi\begin{pmatrix}\textbf{a}\\\textbf{b}\\\textbf{c}\end{pmatrix}=\begin{pmatrix}\textbf{a}\\\textbf{b}\end{pmatrix}$.
\end{prop}

\begin{proof}
Observe that the Gram matrix of the vectors $\{e_1\textbf{a},\cdots, e_7\textbf{a}, e_1\textbf{b},\cdots,$ $e_7\textbf{b}\}$ is given by $\begin{pmatrix}I&\phi\\\phi^t&I\end{pmatrix}$ with $(\phi)_{ij}=:\langle e_i\textbf{a}, e_j\textbf{b}\rangle$, $i, j=1,\cdots, 7$. Clearly, $\phi$ is skew-symmetric and $$\det \begin{pmatrix}I&\phi\\ \phi^t&I\end{pmatrix}=\det(I-\phi^t\phi)\geqslant 0.$$
Simplifying the matrix $G$ defined in (\ref{G}), we find that at the point $B$
\begin{equation*}
\det G=2^{14}\det\begin{pmatrix}
2I&\phi&\phi&0\\\phi^t&2I&0&\phi\\
\phi^t&0&2I&\phi\\
0&\phi^t&\phi^t&2I
\end{pmatrix}=
2^{28}\det\begin{pmatrix}
2I&0
\\-\phi^t\phi&2(I-\phi^t\phi)
\end{pmatrix}
\geqslant 0
\end{equation*}
Thus $B$ is a critical point of $F$ if and only if $\det(I-\phi^t\phi)=0$,
and if and only if $\dim \mathrm{Span}\{\textbf{a}, e_1\textbf{a}, \cdots, e_7\textbf{a}, \textbf{b}, e_1\textbf{b}, \cdots, e_7\textbf{b}\}<16$, which is equivalent to the condition that $A\in \mathrm{Im}~\pi$.
\end{proof}

\begin{rem}
Compared with Proposition \ref{uk},
for $k=4$, we claim $Sp(4)\nsubseteq \mathcal{R}(F)\subset V_4(\mathbb{O}^4)$. In fact, we can choose $U=\frac{1}{\sqrt{2}}\begin{pmatrix} 1&e_1\\ e_2&e_3\end{pmatrix}\in Sp(2)$. Since $U\begin{pmatrix} -e_5\\ e_4\end{pmatrix}=0$, we obtain $U\in \mathrm{Im}~\pi$ of $\pi:~V_3(\mathbb{O}^2)\rightarrow V_2(\mathbb{O}^2)$. Then it follows from Proposition \ref{diag A44} that $\begin{pmatrix} U&0\\0&U \end{pmatrix}\in V_4(\mathbb{O}^4)$ is a critical point of $F$. However, $\begin{pmatrix} U&0\\0&U \end{pmatrix}\in Sp(2)\times Sp(2)\subset Sp(4)$.
\end{rem}

\begin{prop}
Given $A=\begin{pmatrix} \textbf{a}\\ \textbf{b}\end{pmatrix}$,
$B=\begin{pmatrix} \textbf{c}\\ \textbf{d}\end{pmatrix}$
$\in V_2(\mathbb{O}^2)$, the point $C=\begin{pmatrix} A&0\\0&B \end{pmatrix}\in V_4(\mathbb{O}^4)$ is a critical point of $F$ if and only if $\mathrm{det}\big(Z-(\varphi\psi+\psi\varphi)Z^{-1}(\varphi\psi+\psi\varphi)\big)=0$, where
$\varphi=\Phi_{\textbf{ab}}$, $\psi=\Phi_{\textbf{cd}}$ and $Z=4I-(\varphi\varphi^t+\psi\psi^t)$.
\end{prop}
\begin{proof}
Similar to the proof of Proposition \ref{diag A44}, by simplifying the matrix $G$ defined in (\ref{G}), $C$ is a critical point of $F$ if and only if
\begin{equation*}
\det\begin{pmatrix}
4I-(\varphi\varphi^t+\psi\psi^t)&\varphi\psi+\psi\varphi
\\ \varphi\psi+\psi\varphi&4I-(\varphi\varphi^t+\psi\psi^t))
\end{pmatrix}
=0
\end{equation*}
As discussed during the proof of Proposition \ref{diag A44}, we have $I-\varphi\varphi^t\geqslant 0$.
Moreover, $I-\psi\psi^t\geqslant 0$. Hence, $Z$ is positive definite and the conclusion follows directly.
\end{proof}

\end{document}